\documentclass{forArXiv}

\usepackage{amsmath,amsfonts,amssymb}

\usepackage{mathrsfs, hyperref, tikz-cd, hyperref, tabu, mathtools, amsrefs} 

\title{Shimura operators for certain Hermitian symmetric superpairs}                                     %<-------------------
\author{Songhao Zhu}                 %<-------------------
\lastname{Zhu}  %<-------------------

    %<-------------------

         %<-------------------

\address{%
Songhao Zhu\\               %<-------------------
School of Mathematics\\            %<-------------------
Georgia Institute of Technology\\
Atlanta, GA, 30332, USA\\
zhu.math@gatech.edu                %<-------------------
}

\DeclareMathOperator{\Ima}{Im}
\DeclareMathOperator{\Spn}{Span}
\DeclareMathOperator{\Hom}{Hom}
\DeclareMathOperator{\End}{End}

\DeclareMathOperator{\Ind}{Ind}
\DeclareMathOperator{\ad}{ad}
\DeclareMathOperator{\Ad}{Ad}
\DeclareMathOperator{\mult}{\mathit{m}}

\DeclareMathOperator{\diag}{diag}

\DeclareMathOperator{\std}{st}
\DeclareMathOperator{\opp}{op}

\newcommand{\dangle}[1]{\langle #1 \rangle}

\newcommand{\sltwo}{\mathfrak{sl}(2)}
\newcommand{\gl}{\mathfrak{gl}}

\newcommand{\Uk}{\mathfrak{U}}

\newcommand{\gk}{\mathfrak{g}}
\newcommand{\Gk}{\mathfrak{G}}
\newcommand{\Kk}{\mathfrak{K}}
\newcommand{\kk}{\mathfrak{k}}
\newcommand{\tk}{\mathfrak{t}}
\newcommand{\Tk}{\mathfrak{T}}
\newcommand{\hk}{\mathfrak{h}}
\newcommand{\Hk}{\mathfrak{H}}
\newcommand{\bk}{\mathfrak{b}}
\newcommand{\Bk}{\mathfrak{B}}
\newcommand{\ak}{\mathfrak{a}}
\newcommand{\Ak}{\mathfrak{A}}
\newcommand{\pk}{\mathfrak{p}}

\newcommand{\Nk}{\mathfrak{N}}
\newcommand{\Sk}{\mathfrak{S}}
\newcommand{\Pk}{\mathfrak{P}}

\newcommand{\Lk}{\mathfrak{L}}

\newcommand{\C}{\mathbb{C}}
\newcommand{\N}{\mathbb{N}}
\newcommand{\quoUk}{\Uk^\kk/(\Uk\kk)^\kk}
\newcommand{\pplus}{\mathfrak{p}^+}
\newcommand{\pminus}{\mathfrak{p}^-}
\newcommand{\ev}{_{\overline{0}}}
\newcommand{\od}{_{\overline{1}}}
\newcommand{\iso}[1]{{#1^{\texttt{iso}}}}
\newcommand{\ani}[1]{{#1^{\texttt{ani}}}}
\newcommand{\epu}{\epsilon}
\newcommand{\epr}{\alpha\bos}
\newcommand{\dau}{\delta}
\newcommand{\dar}{\alpha\fer}
\newcommand{\bos}{^{\mathrm{B}}}
\newcommand{\fer}{^{\mathrm{F}}}
\newcommand{\bnf}{^{\mathrm{B/F}}}

\newcommand{\zeroS}{\iso{\mathcal{S}}}
\newcommand{\HCso}{\Sigma_\perp}

\newcommand{\supp}[1]{\mathrm{Supp}\left(#1\right)}
\newcommand{\Hyp}{\mathbb{H}}
\newcommand{\RingEv}{\Lambda^\mathtt{0}}
\newcommand{\cH}{\mathscr{H}}

\newcommand{\hcHomoGK}{{\Gamma}}

\newcommand{\Wlambda}{W_\lambda}

\newtheorem{main}{Theorem}

\begin{document}

% Insert title

\maketitle

\begin{abstract}
We give a partial super analog of a result obtained by Sahi and Zhang relating Shimura operators and certain interpolation symmetric polynomials. In particular, we study the pair $(\gl(2p|2q), \gl(p|q)\oplus\gl(p|q))$, define the Shimura operators in $\mathfrak{U}(\gk)^\kk$, and using a new method, prove that their images under the Harish-Chandra homomorphism are proportional to Sergeev and Veselov's Type $BC$ interpolation supersymmetric polynomials under the assumption that a family of irreducible $\gk$-modules are spherical. We prove this conjecture using the notion of quasi-sphericity for Kac modules when $p=q=1$, and give explicit coordinates of (quasi-)spherical vectors.  %<------------------- 
\end{abstract}

\section{Introduction}
In \cite{Shimura90} Shimura introduced a basis for the algebra of invariant differential operators on a Hermitian symmetric space, and formulated the problem of determining their eigenvalues. These eigenvalues can be expressed in terms of the images of these operators under the Harish-Chandra homomorphism, and Shimura's problem was solved in \cite{SZ2019PoSo} where it was shown that such images are specializations of the interpolation polynomials of Type \textit{BC} introduced by Okounkov \cite{Okounkov98}. 
In this paper, we propose the super analog (Theorem~\ref{main:RESULT}) for the Hermitian symmetric superpair $(\gk, \kk)$ with $\gk = \gl(2p|2q)$ and $\kk = \gl(p|q) \oplus \gl(p|q)$. We show that this follows from a conjecture (Conjecture~\ref{conj:sph}) on the $\kk$-sphericity of certain finite dimensional irreducible $\gk$-modules $V_\lambda$. We prove the conjecture when $p=q=1$ (Theorem~\ref{main:11Sph}). 

Let $\Uk$ be the universal enveloping algebra of $\gk$, and $\Uk^\kk$ be the centralizer of $\kk$ in $\Uk$. 
Definition~\ref{defn:ShimuraOp} gives a basis $D_\mu$ (Shimura operators) of $\Uk^\kk$ which involves the Cartan decomposition $\gk =\kk \oplus \pk=\kk \oplus \pplus \oplus \pminus$, and multiplicity-free $\kk$-decompositions of $\Sk(\pplus)$ and $\Sk(\pminus)$ into summands $W_\mu$ and $W^*_\mu$, naturally indexed by $\cH :=\mathscr{H}(p, q)$ of $(p,q)$-hook partitions, which are partitions $\mu$ satisfying $\mu_{p+1}\leq q$.
Next, we consider the Harish-Chandra homomorphism $\hcHomoGK$ defined on $\Uk^\kk$ associated with the pair $(\gk, \kk)$. The kernel of $\hcHomoGK$ is $ (\Uk \kk)^\kk := \Uk\kk\cap \Uk^\kk$ and the quotient $\quoUk$ is isomorphic to the space of invariant differential operators on the underlying symmetric superspace.
Let $\ak \subseteq \pk\ev$ be the maximal toral subalgebra that appears in the Iwasawa decomposition. The image of $\hcHomoGK$ can then be identified with $\RingEv(\ak^*)$, the algebra of even supersymmetric polynomials on $\ak^*$, in which we formulate a suitable specialization of the \textit{Type BC supersymmetric interpolation polynomials} by Sergeev and Veselov in \cite{SV09}. These polynomials $I_\mu$ are characterized by prescribed zeros (\textit{vanishing properties}) and again parameterized by $\cH$.
We give the following answer relating $D_\mu$ and $I_\mu$ for $\gk = \gl(2p|2q)$ and $\kk = \gl(p|q) \oplus \gl(p|q)$. %, assuming a conjecture regarding the sphericity of certain irreducible $\gk$-modules.

\begin{main} \label{main:RESULT}
Assuming \textbf{Conjecture \ref{conj:sph}}, for all $\mu\in \cH$, we have $\hcHomoGK(D_\mu) = c_\mu I_\mu$ for $c_\mu \neq 0$.
\end{main}

We briefly discuss the problem of sphericity here. Let $V_\lambda$ be the irreducible $\gk$-module of the same highest weight of $W_\lambda^*$ with respect to a compatible Borel subalgebra. In the non-super classic scenario, the $\kk$-sphericity of such $V_\lambda$ is completely determined by the  Cartan--Helgason theorem. Basically, a $\gk$-module is spherical if and only if its highest weight vanishes on the $\theta$-fixed part of the Cartan subalgebra and satisfies certain ``even" integrality. The sphericity is crucial in the proof of the result in \cite{SZ2019PoSo}.
Yet in the super setting, such result is only partially known, see \cite{a2015spherical}.

\begin{conjecture}\label{conj:sph}
Every irreducible $\gk$-module $V_\lambda$ for $\lambda \in \cH$ is spherical.
\end{conjecture}

\begin{main}\label{main:11Sph}
\textbf{Conjecture \ref{conj:sph}} is true for $p=q=1$.
\end{main}

Let us point out that this problem has its historical roots in \cites{KS1993, S94} for usual Lie algebras. Various constructions in this project also bear resemblance to other theories, e.g. the super version of the Capelli eigenvalue problem presented in \cite{SSS2020}. 
We hope to fully solve Conjecture~\ref{conj:sph} in future work. 
It is possible that the results of this paper can be generalized to the Hermitian symmetric superpairs constructed using Jordan superalgebras, c.f. \cite{SSS2020}*{Theorem 1.4}.
In \cite{SZ23}, a different approach is applied to circumvent the problem of sphericity. In the current paper, the main ideas are closer to the classic method in \cite{SZ2019PoSo} yet more ``algebraic", highlighted below.

To prove Theorem~\ref{main:RESULT}, we show that $\hcHomoGK(D_\mu)$ satisfies the above characterization of $I_\mu$, namely, (\textit{a}) $\hcHomoGK(D_\mu) \in \RingEv(\ak^*)$, and (\textit{b}) it satisfies the vanishing properties. %These two assertion imply that $\hcHomoGK(D_\mu)$ must be proportional to $I_\mu$.
For (\textit{a}), we check that $\Ima \hcHomoGK$ consists of even supersymmetric polynomials on $\ak^*$. We explain this in Section~\ref{supPairs} (Proposition~\ref{prop:imGammaSupSym}). This follows from the description in \cite{A2012} which is also independently proved in \cite{SZ23} for $\gk = \gl(2p|2q)$ and $\kk = \gl(p|q) \oplus \gl(p|q)$.
%Thus, $\hcHomoGK(D_\mu)$ is such a symmetric polynomial on $\ak^*$.

Proving (\textit{b}) is considerably harder and requires the Conjecture.
We consider the \textit{generalized Verma modules} $M_{-\lambda}$ which are parabolic induced modules and quotients of Verma modules. We show that $M_{-\lambda}$ is cospherical (having a spherical functional) in the ``$\kk$-finite" sense. Then we show that $D_\mu$ annihilates the spherical functionals on $M_{-\lambda}$ and the spherical functionals $\kappa$ on $V_\lambda$ precisely for those $\lambda$ where $I_\mu$ vanishes (Proposition~\ref{prop:Dmu.by0}, c.f. Theorem~\ref{thm:BCPolyRho}). The existence of $\kappa$ is a consequence of the sphericity of $V_\lambda$ (Conjecture~\ref{conj:sph}). This ``vanishing action" of $D_\mu$ leads to (\textit{b}) in the end.
This new technique instantly provides a different approach to the classical problem and sheds light on the potential applications of spherical representations.

%We sketch the argument for Theorem~\ref{main:RESULT} first. 
%Firstly, $D_\mu$ acts on $V_\lambda^\kk$ by a scalar (Lemma~\ref{lem:scalar}). 
%This only requires the fact that $\dim V_\lambda^\kk = 1$ (Proposition~\ref{prop:uniSph}, assuming Conjecture~\ref{conj:sph}). \underCons
%Theorem~\ref{thm:scalar} says the scalar is given by certain evaluation of $\hcHomoGK(D_\mu)$, and Lemma~\ref{lem:branch} says precisely $W_\mu$ does not occur in $V_\lambda$ as a $\kk$-submodule. Theorem~\ref{thm:vanCon} proves the vanishing properties.

We try to solve Conjecture~\ref{conj:sph} in Section~\ref{sphRep} by using the concept of \textit{quasi-spherical} vectors in Kac modules $K(\Breve{\lambda})$. Such vectors descend to spherical vectors in the irreducible quotient.
These $K(\Breve{\lambda})$ are either $V_\lambda$ or have irreducible quotients isomorphic to $V_\lambda$.
For $p = q = 1$, 
we show that $K(\Breve{\lambda})$ is always (quasi-)spherical (Theorems~\ref{thm:b!=a-1} and \ref{thm:b=a-1}) by identifying explicit coordinates of (quasi-)spherical vectors which descend to $V_\lambda$, proving Theorem~\ref{main:11Sph}. This method is purely algebraic.

Our paper is structured as follows.
Section~\ref{prelim} gives an overview of the necessary tools and background results. 
%In the following two sections (Sections~\ref{supPairs} and \ref{sphRep}), we present new formulations and results on the respective subjects. 
Section~\ref{supPairs} gives a concrete construction of the superpair $(\gl(2p|2q), \gl(p|q) \oplus \gl(p|q))$ and the associated restricted root systems. We also discuss the image of Harish-Chandra homomorphism $\hcHomoGK$. 
In Section~\ref{sphRep}, we discuss generalities of spherical representations and specialize the criteria in \cite{a2015spherical}. We then introduce quasi-spherical vectors in Kac modules, with which we prove Theorem~\ref{main:11Sph}. 
Finally, in Section~\ref{pairArg}, we use the generalized Verma modules to prove Theorem~\ref{main:RESULT}. We give an example explaining why \cite{a2015spherical} is insufficient for the vanishing properties. 
In the appendices, we discuss (\textit{1}) admissible deformed root systems and  Sergeev--Veselov polynomials,  (\textit{2}) detailed computations in Section~\ref{sphRep}, and (\textit{3}) a Weyl groupoid invariance formulation of  $\Ima \hcHomoGK$.

\section{Preliminaries} \label{prelim}
In this section, we review the basic definitions and results that will be used later. 
Most results are given without proofs as they can be found in the references.
For the remainder of this paper, we assume that the ground field is $\mathbb{C}$. We include 0 in $\mathbb{N}$. All Lie superalgebras (except for universal enveloping algebras) are assumed to be finite dimensional. We use the upper case $(\Gk, \Kk)$ (including their subspaces) for general discussion and the corresponding lower case frakturs when we specify.

\subsection{The Harish-Chandra Homomorphism} \label{subsec:IDecHCIso}
%\noindent \textbf{2.1. The Harish-Chandra Homomorphism}
We let $(\Gk, \Kk)$ be a symmetric superpair such that the Iwasawa decomposition exists. We first recall some standard terminology for the definition of the Harish-Chandra homomorphism.
The Cartan decomposition $\Gk = \Kk \oplus \Pk$ is given by an (even) involution $\theta$ where $\Kk$ is the fixed point subalgebra and $\Pk$ is the $(-1)$-eigenspace of $\theta$. 
We choose a nondegenerate invariant form $b$ on $\Gk$ that is $\theta$-invariant. 
We assume the existence of toral subalgebra in $\Pk\ev$ and let $\Ak$ be a maximal one. Let $\Sigma := \Sigma(\Gk, \Ak)$ be the restricted root system of $\Gk$ with respect to $\Ak$.
It is well-known that given any root system $\Sigma$, it is possible to choose a positive system $\Sigma^+$ so that $\Sigma = \Sigma^+ \sqcup \Sigma^-$ where $\Sigma^- = -\Sigma^+$ denotes the set of negative roots. 
Let $\Gk_\alpha$ denote the root space of $\alpha\in \Sigma$, and $(-, -)$  the form on $\Ak^*$ induced from $b$.
If $\alpha \in \Sigma$, but $\frac{1}{2}\alpha \notin \Sigma$, we say $\alpha$ is \textit{indivisible}; on the other hand, if $2\alpha \notin \Sigma$, we say it is \textit{unmultipliable} and we denote the set of unmultipliable roots as $\Theta$. We let
%If $\gk _\alpha \cap \gk\ev \neq 0$ (respectively $\gk _\alpha \cap \gk\od \neq 0$), we say $\alpha$ is \textit{even} (respectively odd). We denote the sets of all the even and odd roots as $\Sigma\ev$ and $\Sigma\od$ respectively. 
\[
\Sigma\ev := \{\alpha \in \Sigma : \Gk _\alpha \cap \Gk\ev \neq 0 \}, \; \Sigma\od := \{\alpha \in \Sigma : \Gk _\alpha \cap \Gk\od \neq 0 \}
\]
be the sets of \textit{even} and \textit{odd} roots respectively (subscripts $\overline{0}$ and $\overline{1}$ denote parities). We denote the Weyl group of $\Sigma\ev$ as $W_0$, which is generated by the reflections of even simple roots.
A restricted root may be both even and odd.
A root $\alpha$ is said to be \textit{isotropic} if $( \alpha, \alpha ) = 0$ and \textit{anisotropic} otherwise. We use the superscript \texttt{iso} for isotropic roots.
We set $m_\alpha := \dim (\Gk_\alpha)\ev -\dim (\Gk_\alpha)\od$. The \textit{Weyl vector} is defined as
\[
\rho := \frac{1}{2}\sum_{\alpha \in \Sigma^+} m_\alpha \alpha.
\]
If we set $\rho_i := \frac{1}{2}\sum_{\alpha \in \Sigma^+} \dim(\Gk_\alpha)_i \alpha$ for $i = \overline{0}, \overline{1}$, then $\rho = \rho\ev-\rho\od$.

The Iwasawa decomposition depends on involutions $\theta$ on Lie superalgebras which are classified in \cites{SerganovaRLSA&SSS, SerganovaAuto}. For an in-depth discussion and recent developments on Iwasawa decomposition, see \cite{Sher22Iwasawa}. We define the \textit{nilpotent subalgebra} for $\Sigma^+$ as $\Nk := \bigoplus_{\alpha\in\Sigma^+}\Gk_\alpha$, and
It turns out that we will rely on the ``opposite" Iwasawa decomposition 
\begin{equation} \label{eqn:nak}
    \Gk = \Nk^- \oplus \Ak \oplus \Kk
\end{equation}
with the nilpotent subalgebra $\Nk^-$ for $-\Sigma^+$.
The Weyl vector for \textit{negative} restricted roots is thus $\rho^- = -\rho$. 
The Poincar\'e--Birkhoff--Witt theorem applied to $\Nk^- \oplus \Ak \oplus \Kk$ yields the following identity:
\[
\mathfrak{U} = \left( \mathfrak{U}\Kk + \mathfrak{N^-U}\right ) \oplus \mathfrak{S(A)}.
\]
Following \cite{A2012}, for $D\in \Uk$, we define $\pi(D)\in \mathfrak{S(A)}$ to be the unique element so that $D-\pi (D) \in \Uk\Kk + \mathfrak{N^-U}$,
and define $\hcHomoGK(D)(\lambda) = \pi(D)(\lambda+\rho^-)$ on $\mathfrak{U}^\Kk$. The map $\pi$ is called the \textit{Harish-Chandra projection} and $\hcHomoGK$ the \textit{Harish-Chandra homomorphism}.
As $\rho = -\rho^-$, we have
\begin{equation} \label{eqn:rhoshift}
    \hcHomoGK(D)(\lambda+\rho) = \pi(D)(\lambda).
\end{equation}
Note that this specializes to the well-known Harish-Chandra isomorphism \cite{KacHC} for the ``group case" pair $(\Gk \oplus \Gk, \Gk)$.

In \cite{A2012}, Alldridge introduced certain subalgebra $J_\alpha \subseteq \mathfrak{S}(\Ak)$ for $\alpha \in \Theta\od$ (odd and unmultipliable) used in the following theorem (\cite{A2012}*{Theorem~3.19}).
Let $(\Uk\Kk)^\Kk$ be $\Uk\Kk \cap \Uk^\Kk$ and 
\begin{equation} \label{eqn:Ja}
    J(\Ak) := \mathfrak{S\left(A\right)}^{W_0}  \bigcap_{\alpha \in \Theta \od} J_{\alpha}.
\end{equation}

\begin{theorem} \label{thm:prelHC}
The kernel of $\hcHomoGK$ is $(\Uk\Kk)^\Kk$ and the image of $\hcHomoGK$ is $J(\Ak)$.
\end{theorem}

%The specialization of $J_\alpha$ of the pair of Type \textit{A}III$|$\textit{A}III is postponed to Subsection~\ref{subsec:HCIso}.
When we specialize the discussion to the pair $(\gl(2p|2q), \gl(p|q)\oplus\gl(p|q))$, we are only going to need $J_\alpha$ where $\alpha \in \iso{\Theta\od}$ (odd, unmultipliable and isotropic; see Section~\ref{supPairs}).
In such a case, we can easily describe $J_\alpha$. Theorem~\ref{thm:prelHC} is proved independently using a different method in \cite{SZ23}.

\subsection{Highest Weights and Spherical Representations}\label{subsec:HWtSph} 
%\noindent\textbf{2.2. Highest Weights and Spherical Representations}
Following \cite{musson2012lie}*{Chapters~2, 3, 8  \& 14} (c.f. \cite{CWBook}), we present the conditions on highest weights a finite dimensional irreducible $\Gk$-module, and how they change with respect to different Borel subalgebras in terms of odd reflections.
Then we introduce the results in \cite{a2015spherical} on the highest weights of spherical representations.

Let $\Gk$ be a contragredient Lie superalgebra (see \cite{musson2012lie}*{Hypothesis~8.3.4}), and $\Bk$ be a Borel subalgebra containing a Cartan subalgebra $\Hk$ in $\Gk$. Let $(-, -)$ be the form on $\Hk^*$ that is induced from the one on $\Hk$.
If $\alpha$ is odd non-isotropic, then $2\alpha$ is an even root. Thus it makes sense to set the coroot $h_\alpha$ corresponding to the $\sltwo$-triple of $2\alpha$. Otherwise when $\alpha$ is even, then $h_\alpha$ is defined as usual. We say $\lambda \in \Hk^*$ is \emph{$\Bk$-dominant} if for any non-isotropic simple root $\alpha$ of $\Bk$, $\lambda(h_\alpha) \in \mathbb{N}$. 
%We say $\lambda \in \Hk^*$ is \emph{$\Bk$-dominant} if for any non-isotropic simple root $\alpha$ of $\Bk$, $2(\lambda, \alpha)/(\alpha, \alpha)\in \mathbb{N}$ if $\alpha$ is even, and $2(\lambda, \alpha)/(\alpha, \alpha)\in 2\mathbb{N}$ if $\alpha$ is odd.
Let $V(\lambda, \Bk)$ denote the irreducible $\Gk$-module of $\Bk$-highest weight $\lambda$.

Let $\Sigma_\Bk$ be the set of positive roots of $\Bk$ and $\alpha \in \Sigma_\Bk$ an odd isotropic root. Then $r_\alpha(\Sigma_\Bk) := \left(\Sigma_\Bk \setminus \{\alpha \}\right)\cup \{-\alpha\}$ is a positive system in which $-\alpha$ is simple, and we write $r_\alpha(\Bk)$ for the corresponding Borel subalgebra. We say $\Bk$ and $r_\alpha(\Bk)$ are \emph{adjacent} in this case.
For any two Borel subalgebras with the same even part, there is a sequence of adjacent Borel subalgebras connecting them.
It is also clear that $\Bk\ev = r_\alpha(\Bk)\ev$ and $\Gk_\alpha \subseteq \Bk, \Gk_{-\alpha}\subseteq r_\alpha(\Bk)$. For an odd isotropic root $\alpha$ and $\lambda \in \Hk^*$, we define the \textit{odd reflection of $\alpha$}  on $\lambda$ as follows:
\begin{equation} \label{eqn:oddRef}
    r_\alpha(\lambda)=\begin{cases}
\lambda & \text{ if } \left( \lambda , \alpha\right ) =0\\
\lambda-\alpha & \text{ if }   \left( \lambda,\alpha\right ) \neq 0
\end{cases}.
\end{equation}

The following theorem tells us how to change the highest weight of an irreducible module between two \emph{adjacent} Borel subalgebras. It may be applied consecutively to compute highest weights with respect to different Borel subalgebras. 

\begin{theorem} \label{thm:oddRefWghts}
Let $V$ be an irreducible module of $\Bk$-highest weight $\lambda$, and $\alpha$ be an odd simple root of $\Bk$, then $r_\alpha(\lambda)$ is the $r_\alpha(\Bk)$-highest weight of $V$. That is, $V(\lambda, \Bk) = V(r_\alpha(\lambda), r_\alpha(\Bk))$.
\end{theorem}

\begin{lemma}\label{lem:BorelSeq}
For each simple root $\alpha$ of $\Bk\ev$, there exists a Borel subalgebra $\Bk^\alpha$ obtained from $\Bk$ from a sequence of odd reflections such that either $\alpha$ or $\alpha/2$ is a simple root of $\Bk^\alpha$ and $\Bk^\alpha\ev = \Bk\ev$.
\end{lemma}

\begin{theorem} \label{thm:fDBorelWght}
Let $\Bk^\alpha$ be as above for a simple root $\alpha$ of $\Bk\ev$.
Let $V$ be $V(\lambda^\alpha, \Bk^\alpha)$ where $\lambda^\alpha$ is calculated by consecutively applying (\ref{eqn:oddRef}). Then $\dim V$ is finite dimensional if and only if $\lambda^\alpha$ is $\Bk^{\alpha}$-dominant for all simple roots $\alpha$ of $\Bk\ev$.
%Let $\{\Bk^{(i)}\}$ be the set of Borel subalgebras in view of Lemma~\ref{lem:BorelSeq}. Then $\dim V <\infty$ iff $\lambda^{(i)}$ is $\Bk^{(i)}$-dominant for each $i$.
\end{theorem}

Now following \cite{a2015spherical}, let $(\Gk, \Kk)$ be a reductive symmetric superpair of even type. 
Suppose $\Hk \supseteq \Ak$ is a $\theta$-invariant Cartan subalgebra extended from $\Ak$. 
A $\Gk$-module $V$ is said to be $\Kk$-\textit{spherical} if $V^\Kk := \{v\in V: X.v =0 \textup{ for all }X\in \Kk \}$ is non-zero. A non-zero vector in $V^\Kk$ is called a $\Kk$-\textit{spherical vector}. When the context is clear, we simply say \textit{spherical} instead of $\Kk$-spherical.
We set $\lambda_\alpha := \dfrac{(\lambda|_\Ak, \alpha)}{(\alpha, \alpha)}$ for anisotropic $\alpha \in \Sigma$ as in \cite{a2015spherical}. Given a choice of positive system $\Sigma^+$ of $\Sigma = \Sigma(\Gk, \Ak)$ compatible with a Borel subalgebra $\Bk$,
we say $\lambda \in \Hk^*$ is \textit{high enough} if 
\begin{enumerate}
    \item $(\lambda, \beta )>0$ for any isotropic root $\beta \in \Sigma^+$,
    \item $\lambda_\alpha+m_\alpha+2m_{2\alpha}>0$ and $\lambda_\alpha+m_{\alpha}+m_{2\alpha}+1>0$ for any odd anisotropic indivisible root $\alpha$.
\end{enumerate}
%Then by \cite{a2015spherical}*{Theorem~2.3, Corollary~2.7}, we have the following assertion.
In \cite{a2015spherical} the authors give a sufficient but not necessary condition for $\Kk$-sphericity, partially generalizing the classical Cartan--Helgason theorem, c.f. \cite{Helgason2}*{Theorem~4.1, Chapter~V}.
\begin{theorem}[\cite{a2015spherical}*{Theorem~2.3, Corollary~2.7}] \label{thm:AlldridgeSph}
If $\lambda_\alpha \in \mathbb{N}$ for $\alpha \in \Sigma^+\ev$, $\lambda|_{\Hk\cap \Kk} = 0$, and $\lambda$ is high enough, then $V(\lambda, \Bk)$ is spherical.
\end{theorem}

%This result only provides a sufficient but not necessary condition for $\Kk$-sphericity, partially generalizing the classical Cartan--Helgason theorem, c.f. \cite{Helgason2}*{Theorem~4.1, Chapter~V}.

\subsection{The Cheng--Wang Decomposition} \label{subsec:CWDec}
%\noindent \textbf{2.3 The Cheng--Wang Decomposition}
We now recall a super analog of the Schmid decomposition \cites{Schmid69, FK90} from \cite{CW2000HDfL}. %which implies how the symmetric algebras on $\pplus$ and $\pminus$ decompose. 
We introduce notation first.
A \textit{partition} $\lambda$ is a sequence of non-negative integers $(\lambda_1, \lambda_2, \dots)$ with only finitely many non-zero terms and $\lambda_i\geq \lambda_{i+1}$ (c.f. \cite{MacdSymFun}). 
Let $|\lambda| := \sum_{i}\lambda_i$ denote the \textit{size} of $\lambda$, 
and $\lambda'$ for which $\lambda'_i := |\{j: \lambda_j\geq i\}|$ the \textit{transpose} of $\lambda$. 
A $(p, q)$\textit{-hook partition} is a partition $\lambda$ such that $\lambda_{p+1}\leq q$. We write
\[
\cH = \mathscr{H}(p, q) := \left \{\lambda: \lambda_{p+1}\leq q \right \}, \quad  \cH^d = \mathscr{H}^d(p, q) := \left \{\lambda\in \mathscr{H}(p, q):|\lambda|=d \right \}.
\]
For $\lambda \in \mathscr{H}(p, q)$, we define a $(p+q)$-tuple
\begin{equation}\label{eqn:lNat}
    \lambda^\natural := \left(\lambda_1, \dots, \lambda_p,  \langle \lambda_1'-p \rangle, \dots, \langle \lambda_q'-p \rangle \right) 
\end{equation}
where $\langle x \rangle := \max \{x, 0\}$ for $x\in \mathbb{Z}$.
The last $q$ coordinates can be viewed as the lengths of the remaining columns after discarding the first $p$ rows of $\lambda$.

We let $E_{i, j}$ denote the standard matrix with 1 in the $(i, j)$-th entry, and 0 elsewhere. For $\gl(a|b)$, let $\tk = \{E_{i, i}: i = 1, \dots, a+b\}$ be the standard diagonal Cartan subalgebra, and set $\epsilon_i$ for $i = 1, \dots, a$ and $\delta_j$ for $j = 1, \dots, b$ as the standard coordinates of $\tk$ from top left to bottom right.
Let $\mathfrak{B}$ be any Borel subalgebra of $\gl(p|q)$ containing a Cartan subalgebra $\mathfrak{H}$. Then $\mathfrak{B}$ can be described by an $\epsilon\delta$ \textit{-chain}, $[X_1\cdots X_{p+q}]$, a sequence consisting of characters $X_i \in \mathfrak{H}^*$ 
where $X_i-X_{i+1}$ exhaust all simple roots defining $\Bk$. 
Therefore, the following two $\epsilon\delta$-chains
\[
[\epu_1\cdots\epu_p | \dau_1\cdots\dau_q], \quad [\dau_q\cdots\dau_1 | \epu_p\cdots\epu_1]
\]
respectively correspond to the standard Borel subalgebra $\bk^{\std}$ and the opposite one $\bk^{\opp}$.
For an $(p+q)$-tuple $(a_1, \dots, a_p | b_1, \dots, b_q)$, we associate an irreducible $\gl(p|q)$-module of highest weight $\sum_{i=1}^p a_i\epu_i + \sum_{j = 1}^q b_j \dau_j$ with respect to $\bk^{\std}$, denoted as $L(a_1, \dots, a_p, b_1, \dots, b_q)$. 

Let $V$ be a vector superspace and $\Sk(V)$ be the supersymmetric algebra on $V$. Then $\Sk(V)$ has a natural $\N$-grading $\bigoplus_{k\in\N}\Sk^n(V)$. Explicitly, $\Sk^d(V) = \bigoplus_{i+j = d}\Sk^i(V\ev)\otimes \bigwedge^j(V\od)$ as vector spaces.
The natural action of $\gl(p|q)$ on $\C^{p|q}$ gives an action of $\gl(p|q)\oplus \gl(p|q)$ on $\C^{p|q}\otimes \C^{p|q}$, which extends to an action on $\Sk(\C^{p|q}\otimes \C^{p|q})$. We record the following result regarding the $\gl(p|q)\oplus \gl(p|q)$-module structure on $\Sk(\C^{p|q}\otimes \C^{p|q})$. See \cite{CW2000HDfL}*{Theorem~3.2}.
\begin{theorem} \label{cite:CW}
    The supersymmetric algebra $\Sk(\C^{p|q}\otimes \C^{p|q})$ as a $\gl(p|q)\oplus \gl(p|q)$-module is completely reducible and multiplicity-free. In particular,
    \[
    \Sk^d(\C^{p|q}\otimes \C^{p|q}) = \bigoplus_{\lambda\in \mathscr{H}^d} L(\lambda^\natural)\otimes L(\lambda^\natural).
    \]
\end{theorem}

\subsection{Type \textit{BC} Interpolation Polynomials} \label{subsec:sPoly}
Finally, we introduce a key result in Sergeev and Veselov's paper regarding the interpolation supersymmetric polynomials. Specifically, we need a suitable version of Proposition 6.3 in \cite{SV09}, c.f. \cite{OO06}. We will point out how one may obtain this formulation using \cite{SV09} in Appendix \ref{app:GRS}. Set $i = 1, \dots, p$ and $j = 1, \dots, q$.
Let $\{e_i\}\cup\{d_j\}$ be the standard basis for $V = \mathbb{C}^{p+q}$ with coordinates $x_i$ and $y_j$.

\begin{Definition}\label{defn:noshiftBCSym}
A polynomial $f \in \mathfrak{P}(V)= \mathbb{C}[x_i, y_j]$ is said to be \textit{even supersymmetric} if :
    \begin{enumerate}
        \item[(i)] $f$ is symmetric in $x_i$ and $y_j$ separately and invariant under sign changes of $x_i$ and $y_j$;
        \item[(ii)] $f(X+e_i-d_j)=f(X)$ if $x_i+y_j = 0$.%where $X = \sum x_i e_i +\sum  y_j d_j$.
    \end{enumerate}
We denote the subring of even supersymmetric polynomials as $\RingEv(V)$.
\end{Definition}

We denote $\sum \lambda_i e_i + \sum   \langle \lambda_j'-p  \rangle d_j$ by $\lambda^\natural$, c.f. (\ref{eqn:lNat}). Fixing $p, q$, we write $\cH$ for $\cH(p, q)$.
Set
\begin{equation} \label{eqn:genRho}
    \rho := \sum (2(p-i)+1-2q) e_i + \sum (2(q-j)+1) d_j.
\end{equation}
We are now ready to introduce the Type $BC$ interpolation supersymmetric polynomials in $\RingEv(V)$.

\begin{theorem}\label{thm:BCPolyRho}
For each $\mu \in \cH$, there exists a unique degree $2|\mu|$ polynomial $I_\mu \in \RingEv(V)$ such that
\[
I_\mu(2\lambda^\natural + \rho) = 0 
\]
for any $\lambda \in \cH$ such that $|\lambda| \leq |\mu|$, $\lambda \neq \mu$, and satisfies the normalization condition
\[
I_\mu(2\mu^\natural + \rho) = \prod_{(i, j)\in \mu}\left(\mu_i -j+\mu'_j -i +1 \right)\left(\mu_i +j-\mu'_j -i +2p-2q \right).
\]
Moreover, they constitute a basis for $\RingEv(V)$.
\end{theorem}

The coordinates of $\rho$ come from the restricted root system in Section~\ref{supPairs}, and the appearance of 2 in the above vanishing properties is also a result of choice of coordinates which is non-essential but convenient.
For a detailed discussion on how we obtain the above formulation, see Appendix~\ref{app:GRS}.
%The proof relies upon (1) a special map that transfers the data from the ring of even symmetric polynomials to $\Lambda^\varrho$, (2) an even-Bernoulli-polynomial argument, and (3) explicit specification of parameters in the original definition in \cite{SV09} and a change of variables. 

%{\color{red} Issue: bad choices of the letters?}

%Again, we use the superindex $\varrho$ to indicate a $\varrho$ shift in their paper. In order to keep our paper in line with \cite{SZ2019PoSo}, we make a change of variable $X \mapsto X+\varrho$

\section{Realizations of the Symmetric Superpairs and Root Data}\label{supPairs}
From now on, we fix $\Gk = \gk = \gl(2p|2q)$ and $\Kk = \kk =  \gl(p|q)\oplus \gl(p|q)$. We first describe the pair $(\gk, \kk)$ and important subspaces therein which are specified as the lower case counterparts of those appearing in Section~\ref{prelim}.
We then define Shimura operators and give specific root data.
Finally, we formulate the image of $\hcHomoGK$ as algebra of even supersymmetric polynomials.

\subsection{Realization of \texorpdfstring{$(\gk, \kk)$}{(g, k)}} \label{subsec:2Cartan}
We fix the following embedding of $\kk$ into $\gk$
\begin{equation}\label{eqn:emb}
\left( \left( \begin{array}{c|c}A_{p \times p} & B_{p \times q}\\ \hline C_{q \times p} & D_{q \times q}\end{array}\right ),\left( \begin{array}{c|c}A_{p \times p}' & B_{p \times q}'\\ \hline C_{q \times p}' & D_{q \times q}'\end{array} \right ) \right ) 
\mapsto 
\left( \begin{array}{c c |c c } A_{p \times p} & 0_{p \times p} & B_{p \times q} & 0_{p \times q} \\0_{p \times p} & A_{p \times p}' & 0_{p \times q} & B_{p\times q}' \\ \hline C_{q \times p} & 0_{q \times p} & D_{q \times q} & 0_{q \times q} \\ 0_{q \times p} & C_{q \times p}' & 0_{q \times q} & D_{q \times q}'  \end{array}\right ),
\end{equation}
and identify $\kk$ with its image.
We let $J := \frac{1}{2}\diag(I_{p \times p}, -I_{p \times p}, I_{q \times q}, -I_{q \times q})$, and $\theta := \Ad \exp(i\pi J)$. 
Then $\theta$ has fixed point subalgebra $\kk \cong \gl(p|q)\oplus \gl(p|q)$. 
We also have the \textit{Harish-Chandra decomposition}
\[
\gk = \pminus \oplus \kk \oplus \pplus
\]
where $\mathfrak{p}^+$ (respectively $\mathfrak{p}^-$) consists of matrices with non-zero entries only in the upper right (respectively bottom left) sub-blocks in each of the four blocks, that is, non-zero entries of matrices in $\pplus$ (respectively $\pminus$) are right to (respectively below) $X$ and above (respectively left to) $X'$ for $X = A, B, C, D$ in (\ref{eqn:emb}), indices omitted. 
Note $\pminus \oplus \kk \oplus \pplus$ is a grading by eigenvalues $(-1, 0, 1)$ of $\ad J$, also known as a \textit{short grading}.
This pair is an example of a Hermitian symmetric superpair which always admits the Harish-Chandra decomposition due to the existence of a complex structure, c.f. \cite{CFherm}.
Set $\pk := \pminus\oplus \pplus$, so $\gk = \kk \oplus \pk$ with $\theta|_\pk = -\mathrm{Id}_\pk$.

In our theory, we need to work with a $\theta$-stable Cartan subalgebra $\hk$ of $\gk$ extended from the toral subalgebra $\ak \subseteq \pk\ev$. 
Note in \cite{AHZ2010Crtf}*{Section~4}, it is showed that $(\gl(r+p|s+q), \gl(r|s)\oplus\gl(p|q))$ is of even type if and only if $(r-p)(s-q)\geq 0$, satisfied by our choice $p = r, q = s$.
We present a construction of $\hk$ and $\ak$ using a certain Cayley transform as follows. 
Set $i = 1, \dots, p$ and $j = 1, \dots, q$.

Recall that $\tk$ denotes the diagonal Cartan subalgebra. On $\gl(2p|2q)$, we set $\epu_i^+:= \epu_i$, $\epu_i^-:=\epu_{p+i}$, $\dau_j^+:=\dau_j$, and $\dau_j^-:=\dau_{q+j}$ on $\tk$. 
Let
\[
\gamma_i\bos := \epu_i^+-\epu_i^-, \quad \gamma_j\fer:= \dau_j^+-\dau_j^-.
\]
These are the \textit{Harish-Chandra strongly orthogonal roots}. We further denote $\HCso := \{\gamma_k\bnf\}$ \footnote{Here $\mathrm{B}$ indicates the Boson--Boson block (top left) and $\mathrm{F}$ the Fermion--Fermion block (bottom right), c.f. \cite{A2012}.}.
We set $A_{i, i'} := E_{i, i'}$ for $1 \leq i, i' \leq 2p$ and $D_{j, j'} := E_{2p+j, 2p+j'}$ for $1 \leq j, j' \leq 2q$.
Associated with each $\gamma\bos_i$ is an $\sltwo$-triple spanned by $A_{i, i}-A_{p+i, p+i}, A_{i, p+i}$ and $A_{p+i, i}$ (similarly for $\gamma\fer_j$ with $D_{j, j'}$). It is not hard to see that all $(p+q)$ $\sltwo$-triples commute. 
Define 
\begin{align*}
    c\bos_i &:= \Ad \exp\left( {\frac{\pi}{4}}\sqrt{-1}(-A_{i, p+i} - A_{p+i, i} )\right), \\
    c\fer_j &:= \Ad \exp\left( {\frac{\pi}{4}}\sqrt{-1}(-D_{j, q+j} - D_{q+j, j} )\right).
\end{align*}
The product
\begin{equation} \label{eqn:Cayley}
    c := \prod_i c\bos_i \prod_j c\fer_j 
\end{equation}
is thus a well-defined automorphism on $\gk$ as all terms commute. 
We set 
\begin{alignat*}{5}
    x_i &:= \sqrt{-1}(A_{i,p+i} - A_{p+i, i}), &&\quad x'_i &&:= A_{i,i}+A_{p+i, p+i},  &&\quad x_{\pm i} &&:= \frac{1}{2}(x'_i\pm x_i); \\
    y_j &:= \sqrt{-1}(D_{j, q+j} - D_{q+j,j}),
    && \quad  y'_j &&:= D_{j,j}+D_{q+j, q+j},
    && \quad  y_{\pm j} &&:= \frac{1}{2}(y'_j\pm y_j).
\end{alignat*}
Then by a direct (rank 1) computation, we see that under $c$:
\begin{equation} \label{eqn:rk1}
    A_{i,i} \mapsto x_{+i}, \quad A_{p+i, p+i}\mapsto x_{-i}, \quad D_{j, j} \mapsto y_{+j}, \quad D_{q+j, q+j} \mapsto y_{-j}.
\end{equation}
We define the following subspaces:
\begin{equation} \label{eqn:hCartan}
    \hk := c(\tk), \quad \ak := \Spn_\C \{x_i, y_j\}, \quad \tk_+ := \Spn_\C \{x'_i, y'_j\}. \quad (\text{Note }\hk = \ak\oplus \tk_+.)
\end{equation}
In $\tk$, the space $\tk_- := \Spn_\C \{A_{i,i}-A_{p+i, p+i}, D_{j,j} - D_{q+j, q+j}\}$ is the orthogonal complement of $\tk_+$ with respect to the Killing form on $\gk$ so we have $\tk = \tk_+ \oplus \tk_-$.
Also, on $\hk$ we let
\[
\epr_i, \dar_j, \tau_i\bos, \tau_j\fer \in \hk^*
\]
be dual to $x_i, y_j, x'_i, y'_j$ respectively. 
Then $\tau_k\bnf$ and $\alpha_k\bnf$ vanish on $\ak$ and $\tk_+$ respectively. 
We identify $\alpha_k\bnf$ with its restriction to $\ak$.
We also have
\begin{equation} \label{eqn:halfRel}
    \tau_i\bos = \frac{1}{2}(\epu_i^++\epu_i^-), \quad \tau_j\fer = \frac{1}{2}(\dau_j^++\dau_j^-)
\end{equation}
and we identify them with their restrictions on $\tk_+$.
On $\ak^*$, we set
\begin{equation}\label{eqn:akForm}
    (\alpha_m\bos, \alpha_n\bos)=-(\alpha_m\fer, \alpha_n\fer) = \delta_{mn}, \quad
    (\alpha_i\bos, \alpha_j\fer)=0,
\end{equation}
for $1\leq m, n \leq p$ or $q$.
For future purposes,  we also give a basis for $\hk^*$ defined by
\begin{equation} \label{eqn:chiEta}
    \chi_{\pm i} := \tau_i\bos \pm \alpha_i\bos, \quad 
    \eta_{\pm j} := \tau_j\fer \pm \alpha_j\fer.
\end{equation}

\subsection{Shimura operators} \label{subsec:SOp}
Recall that we set $i = 1, \dots, p$ and $j = 1, \dots, q$, and that the irreducible $\gl(p|q)$-module $L(\lambda^\natural)$ has highest weight 
\[
\sum\lambda_i \epu_i+\sum\langle \lambda'_j-p\rangle \dau_j
\]
with respect to $\bk^{\std}$, and is of Type \texttt{M} (\cite{CWBook}). In this case, Schur's Lemma says $\dim \End_{\gl(p|q)}\left(L(\lambda^\natural)\right) = 1$, and $L(\lambda^\natural) \otimes L(\lambda^\natural)$ is irreducible as a $\kk$-module for $\kk = \gl(p|q)\oplus \gl(p|q)$. 
If we let $\gl(p|q)$ act on the second component contragrediently (via negative supertranspose), then we define the irreducible $\kk$-module $L(\lambda^\natural) \otimes L^*(\lambda^\natural)$ as $\Wlambda$. 
Note both $\pminus$ and $\pplus$ are $\kk$-modules by the $(-1, 0, 1)$-grading.
We identify $\pminus$ as $(\pplus)^*$ via the supertrace form, which is non-degenerate. 

\begin{proposition} \label{prop:goodCW}
The symmetric superalgebras $\mathfrak{S}(\pplus)$ and $\mathfrak{S}(\pminus)$ are completely reducible and multiplicity free as $\kk$-modules. Specifically,
    \begin{equation} \label{eqn:goodCW}
    \mathfrak{S}^d(\mathfrak{p}^+) = \bigoplus_{\lambda\in \mathscr{H}^d(p, q)} \Wlambda,\quad \mathfrak{S}^d(\mathfrak{p}^-) = \bigoplus_{\lambda\in \mathscr{H}^d(p, q)} \Wlambda^*.
\end{equation}
\end{proposition}
\begin{proof}
    By duality it suffices to show the first equation. 
    First we have $\pplus\cong \C^{m|n}\otimes(\C^{m|n})^*$, by identifying $\C^{m|n}$ and $(\C^{m|n})^*$ as spaces of column and row vectors respectively. 
    The contragredient $\kk$-module structure on $(\C^{m|n})^*$ is obtained by applying the negative supertranspose on $\gl(p|q)$.
    Then Theorem~\ref{cite:CW} implies
    \[
    \Sk^d(\pplus) = \bigoplus_{\lambda\in \mathscr{H}^d(p, q)} L(\lambda^\natural)\otimes L^*(\lambda^\natural),
    \]
    proving the claim, c.f. \cite{SSS2020}*{Theorem~1.4} and the notation therein.
\end{proof}

The highest weight of $\Wlambda$ with respect to the Borel subalgebra $\bk^{\std}\oplus \bk^{\opp}$ of $\kk$ is given by:
\begin{align}
     \tk^* \ni \lambda^\natural_\tk &:= \sum \lambda_i  \epu^+_i + \sum\langle \lambda_j'-p \rangle \dau^+_j -\sum \lambda_i  \epu^-_i - \sum \langle \lambda_j'-p \rangle \dau^-_j \label{eqn:naturalWt} \\
     &= \sum \lambda_i\gamma_i\bos + \sum \langle \lambda_j'-p \rangle \gamma_j\fer, \quad \gamma_i\bos, \gamma_j\fer \in \HCso.\label{eqn:gammaWt}
\end{align}
The weight $\lambda_\tk^\natural$ is indeed dominant since we choose $\bk^{\opp}$ for the second copy of $\gl(p|q)$ in $\gl(p|q)\oplus \gl(p|q)$. Such a choice of positivity appears in \cite{SSS2020}, and also in \cite{FZ21} in a different context.
We also record the highest weight of $\Wlambda^*$ with respect to $\bk^{\opp}\oplus \bk^{\std}$:
\begin{equation}
    -\lambda^\natural_\tk
= \sum -\lambda_i\gamma_i\bos - \sum \langle \lambda_j'-p \rangle \gamma_j\fer.\label{eqn:-gammaWt}
\end{equation}

As $\pk^\pm$ are supercommutative, the respective universal enveloping algebras are just $\Sk(\pk^\pm)$.
Since $\pminus$ is identified with $(\pplus)^*$, 
the direct summand $\Wlambda^*\otimes \Wlambda$ is embedded in $\mathfrak{S}(\pminus) \otimes\mathfrak{S}(\pplus)$ and then multiplied into $\Uk$. 
We write $1_\lambda$ for the element in $\left(\Wlambda^*\otimes \Wlambda\right)^\kk$ corresponding to $\mathrm{Id}_{\Wlambda}\in \End_\kk(\Wlambda)$ under the natural isomorphism.
\begin{Definition} \label{defn:ShimuraOp}
For each $\lambda \in \mathscr{H}(p, q)$, we let $D_\lambda$ be the image corresponding to $1_\lambda$ under the composition of the multiplication and the embedding
\begin{align} \label{eqn:longArrow}
        \left(\Wlambda^*  \otimes \Wlambda \right)^\kk \hookrightarrow \left( \Sk(\pminus)\otimes \Sk(\pplus) \right)^\kk  \rightarrow  &\Uk^\kk \notag \\
        1_\lambda  \xmapsto{\hphantom{ 
        \left(\Wlambda^* \otimes \Wlambda \right)^\kk \hookrightarrow \left( \Sk(\pminus)\otimes \Sk(\pplus) \right)^\kk  \rightarrow  \Uk^\kk
        } }  
        & D_\lambda.
\end{align}
The element $D_\lambda$ is called the \textit{Shimura operator associated with the partition }$\lambda$.
\end{Definition}

\subsection{Restricted Root Data} \label{resRS}
In order to study the symmetry associated with a symmetric superpair, we need to know its restricted root system. We now give the restricted root data of $\gk = \gl(2p|2q)$ and $\kk =  \gl(p|q)\oplus \gl(p|q)$. 
%Despite this, we present tables of restricted root data for various cases. 
%\begin{remark} 
%Our specific constructions for $\ak$ are almost always diagonal/anti-diagonal. See Appendix \ref{app:supPairs} for a detailed description.
%\end{remark}
%Just as in the non-restricted case, one may define the concept of the parity of roots:
%\[
%\Sigma\ev:= \left \{\alpha \in \Sigma : \gk _\alpha \cap \gk\ev \neq 0 \right \},\Sigma\od:= \left \{\alpha \in \Sigma : \gk _\alpha \cap \gk\od \neq 0 \right \}.
%\]
%There is no guarantee that they are disjoint; in fact, \hyperref[tab:2]{Table 3} indicates that they intersect non-trivially in various cases. Still, $\Sigma = \Sigma\ev \cup \Sigma\od$. 
At this point, it is desirable to define the \textit{deformed multiplicity} of a root in $\Sigma := \Sigma(\gk, \ak)$:
\begin{equation} \label{eqn:deformedMult}
\mult (\alpha) := -\frac{1}{2}m_\alpha =  -\frac{1}{2}\left( \dim (\gk_\alpha)\ev -\dim (\gk_\alpha)\od \right).
\end{equation} 
In the following table, we give the superdimensions of the restricted root spaces together with the corresponding $m(\alpha)$ denoted as the parameters $\mathsf{p, q, k, r, s}$.
Note only the last column gives the odd restricted roots and clearly $\Theta\od = \iso{\Theta\od}$.

\begin{center}
    \begin{tabu}{|X[0.8, c, m]|X[c, m]|X[c, m]|X[0.8, c, m]|X[c, m]|X[1.2, c, m]|X[c, m]|}
    \hline
     $\epr_i,\mathsf{ p}$ & $2\epr_i, \mathsf{q}$ & $\epr_i\pm\epr_j, \mathsf{k}$ & $\dar_i, \mathsf{r}$ & $2\dar_i, \mathsf{s}$  &  $\dar_i\pm\dar_j, \mathsf{k}^{-1}$ & $\epr_i\pm\dar_j, 1$  \\
    \hline
     $0|0, 0$ & $1|0, -\frac{1}{2}$ & $2|0, -1$ & $0|0, 0$ & $1|0, -\frac{1}{2}$ & $2|0, -1$ & $0|2, 1$\\
    \hline
    \end{tabu}
Table: Positive restricted roots and deformed multiplicities
\end{center}
The Weyl group $W_0$ associated with all the above cases is of \textit{Type} \textit{BC}; explicitly, $W_0$ is $(\mathscr{S}_p \ltimes (\mathbb{Z}/2\mathbb{Z})^p)\times (\mathscr{S}_q \ltimes (\mathbb{Z}/2\mathbb{Z})^q)$, the direct product of two usual Type $BC$ Weyl groups, permuting and alternating the signs of $\alpha\bos$ and $\alpha\fer$ separately,
c.f. Definition~\ref{defn:LambdaRing} and Definition~\ref{defn:noshiftBCSym}.
It is not hard to verify that all the above $\mathsf{k, p, q, r, s}$ satisfy the two relations
\[
\mathsf{p = kr}, \mathsf{2q+1 = k(2s+1)}
\]
as in \cite{SV09}. These parameters are used in defining \textit{deformed root systems} (see Appendix~\ref{app:GRS}).
We also record the Weyl vector positive restricted roots:
\begin{equation} \label{eqn:AiRho}
    \rho = \sum_{i=1}^p (2(p-i)+1-2q) \epr_i +\sum_{j=1}^q (2(q-j)+1)\dar_j
\end{equation}
and $\varrho := \frac{1}{2}\sum_{\alpha\in\Sigma} \mult(\alpha)\alpha$ is therefore $-\frac{1}{2}\rho$ by (\ref{eqn:deformedMult}) (c.f. (\ref{eqn:genRho})).
%In particular, the hyperplane equation degenerates to $(X-\varrho, \alpha) =0$ as the inner product specifies $\mathsf{k}=-1$. 

\subsection{Highest Weights and Cayley Transforms}  \label{subsec:Cayley}
To relate the two different choices of Cartan subalgebra, $\tk$ and $\hk$, we consider the Cayley transform $c$ (see (\ref{eqn:Cayley})) which allows us to send the positivity on one Cartan subalgebra to that of the other Cartan subalgebra. Specifically, we have $c^{-1} : \hk \rightarrow \tk$, and the dual map $c^{-1}_*: \tk^* \rightarrow \hk^*$. 
It is then a direct computation to check that 
\begin{equation}\label{eqn:Cayley*}
    c^{-1}_*: \epu^\pm_i \mapsto \chi_{\pm i}, \quad \dau^\pm_j \mapsto \eta_{\pm j}.
\end{equation}
First, we set the positive roots $\Sigma^+(\gk, \tk)$ according to the chain (c.f. Section~\ref{prelim}) 
\begin{equation}\label{eqn:chain}
    [\epu_1^+ \cdots \epu_p^+ | \dau_1^+ \cdots \dau_q^+ \dau_q^- \cdots \dau_1^- | \epu_p^-  \cdots \epu_1^-]
\end{equation}
which equivalently gives a Borel subalgebra of $\gk$, denoted as $\bk^+_\natural$. 
In fact, we have
\begin{equation} \label{eqn:bNatural}
    \bk_\natural^+ = \bk^{\std}\oplus \bk^{\opp}\oplus \pplus.
\end{equation}
By applying $c^{-1}_*$ to each of the terms in a chain, we can transfer the positivity corresponding to one Borel subalgebra to another. Thus, the $c^{-1}_*$ induced choice of positivity of $\Sigma(\gk, \hk)$ is given by
\begin{equation} \label{eqn:hkChain}
    [\chi_{+1} \cdots \chi_{+p} | \eta_{+1} \cdots \eta_{+q}\eta_{-q} \cdots \eta_{-1} | \chi_{-p} \cdots \chi_{-1}].
\end{equation}
We denote the corresponding Borel subalgebra as $\bk^+$.
Specifically, when restricted to $\ak^*$, $2\epr_i, 2\dar_j$, $\epr_i \pm \epr_{i'}$, $\dar_j \pm \dar_{j'}$, and $\epr_i \pm \dar_{j}$ are positive for $\Sigma^+ = \Sigma^+(\gk, \ak)$.
%Let $\bk^{\ell}$ denote the Borel subalgebra corresponding to a chain $\ell$, and $c(\bk^{\ell})$ denote the Borel subalgebra corresponding to the chain $c^{-1}_*(\ell)$. For example, if $\ell = $(\ref{eqn:chain}) then $c^{-1}_*(\ell)$ gives (\ref{eqn:hkChain}).
Similarly, the opposite choices are: 
\begin{equation}\label{eqn:b-natural}
    [\epu_1^- \cdots \epu_p^- | \dau_1^- \cdots \dau_q^- \dau_q^+ \cdots \dau_1^+ | \epu_p^+  \cdots \epu_1^+] \longleftrightarrow \bk_\natural^- = \bk^{\opp}\oplus \bk^{\std}\oplus \pminus.
\end{equation}
Then $c^{-1}_*$ induces the following chain and Borel subalgebra of $\gk$:
\begin{equation} \label{eqn:b-}
    [\chi_{-1} \cdots \chi_{-p} | \eta_{-1} \cdots \eta_{-q}\eta_{+q} \cdots \eta_{+1} | \chi_{+p} \cdots \chi_{+1}] \longleftrightarrow \bk^-.
\end{equation}
Now $-2\epr_i, -2\dar_j$, $-\epr_i \mp \epr_{i'}$, $-\dar_j \mp \dar_{j'}$, and $-\epr_i \mp \dar_{j}$ are positive in $\Sigma^- = \Sigma^-(\gk, \ak)$. We will use this choice for generalized Verma modules in Section~\ref{pairArg}.

In order to write down the Cayley transform on a weight module of $\gk$, let us look at a ``rank 1" computation in the same spirit of the discussion in Section~\ref{supPairs}. 
Define 
\[\epsilon^\pm: \begin{pmatrix}
x^+ & 0 \\
0 & x^-\\
\end{pmatrix}\mapsto x^\pm, \text{ and } \alpha: A\mapsto a \text{ for } A := \begin{pmatrix}
0 & ia \\
-ia & 0\\
\end{pmatrix}
\]
and set $\gamma:= \epsilon^+- \epsilon^-$. Let $\{X, H, Y\}$ be the standard $\sltwo$-triple (see Appendix \ref{app:11brackets}). Then $\gamma(H) = 2$.
We let $c_0 := {-\sqrt{-1}\frac{\pi}{4}} (X+Y)$, and
$C_0 := \exp{c_0}\in SL(2)$. 
Suppose $(\phi, \bigoplus_{n\in \mathbb{Z}}V_{n\gamma})$ is a finite dimensional weight module of $\sltwo$. Then there is an $SL(2)$ action $\Phi$ on $V$ so that $\phi = d\Phi$ and $\Phi \circ \exp = \exp\circ \phi$.
\begin{lemma} \label{eqn:CayleyWeight}
    Let the notation be as above. Let $v\in V_{n\gamma}$ and $w = \Phi(C_0) v$. Then $\phi(A)w = 2n\alpha(A)w$.
\end{lemma}
\begin{proof}
We compute the action of $A$ on $w$:
\begin{alignat*}{2}
    \phi(A)w &= \Phi(C_0)\Phi(C_0^{-1})\phi(A)w \\
    &= \Phi(C_0)\exp(\phi(-c_0))\phi(A)\exp(\phi(c_0)) v \quad  && [\Phi \circ \exp = \exp\circ \phi]\\
    &= \Phi(C_0)\phi(\Ad \exp(-c_0)A) v  \, && [\text{Adjoint}]\\
    &= \Phi(C_0)\phi\left( \begin{pmatrix}
    a & 0 \\
    0 & -a\\
    \end{pmatrix}\right) v. && [\text{Matrix exponential}]
\end{alignat*}
The last expression is just 
$2na \Phi(C_0) v= 2n\alpha(A)w$, i.e. $w$ has $A$-weight $2n\alpha$.
\end{proof}

\begin{proposition} \label{prop:even}
Let $V(\lambda_\tk, \bk_\natural^\pm)$ be a finite dimensional highest weight $\gk$-module of highest weight $\lambda_\tk = \sum n_i\bos\gamma_i\bos + \sum n_j\fer \gamma_j\fer$. 
Then $V(\lambda_\tk, \bk_\natural^\pm) = V(\lambda_\hk, \bk^\pm)$ where $\lambda_\hk = \sum 2n_i\bos\epr_i + \sum 2n_j\fer \dar_j$. In particular, $\lambda_\hk\vert_{\tk_+} = 0$.
\end{proposition}
\begin{proof}
Apply the above rank-1 reduction repeatedly to the highest weight space of $V(\lambda_\tk, \bk^\natural)$ by restricting the action of $\gk$ to the corresponding $\sltwo$-triple. 
Since (\ref{eqn:Cayley}) is well-defined, so is this process. 
Indeed, $n_k\bnf\gamma_k\bnf$ becomes $2n_k\bnf\alpha_k\bnf$ as a weight by Lemma~\ref{eqn:CayleyWeight}.
\end{proof}

\subsection{Image of \texorpdfstring{$\hcHomoGK$}{Gamma}} \label{subsec:imGamma}
Finally, we specialize $J(\ak)$ (see (\ref{eqn:Ja})), the image of $\hcHomoGK$, in Theorem~\ref{thm:prelHC} to $(\gl(2p|2q), \gl(p|q)\oplus\gl(p|q))$ and show that it consists of even supersymmetric polynomials.
Note by the root data above, $\Theta\od = \iso{\Theta\od} =
 \{ \pm \alpha_i\bos \pm \alpha_j\fer \}$.
For our $\alpha$, $J_\alpha$ affords an explicit description.
Let $\alpha \in \Theta\od$. Choose $h_0 \in \ak$ such that $\alpha( h_0) = 1, b \left(h_0, h_0 \right) = 0$. Let $t_\alpha \in \ak$ be unique such that $b(t_\alpha, a) = \alpha(a)$ for all $a\in \ak$. Note $b(h_0, t_\alpha)=1$.
Let $\mathfrak{a_\alpha}$ be $\Spn_\mathbb{C} \left \{ h_0, t_\alpha \right \}$, and $\mathfrak{a^{\perp}_\alpha }$ be the orthogonal complement of $\mathfrak{a_\alpha}$ in $\mathfrak{a}$. Then $\ak = \ak_\alpha \oplus\ak_\alpha^\perp$ since $b$ is non-degenerate when restricted to $\ak_\alpha$.
In \cite{A2012}*{Notation~3.7}, the subalgebra $J_\alpha$ is defined as $I_{\alpha, \mathfrak{m}_\alpha}\mathfrak{S\left ( a^{ \perp}_\alpha \right )}$ where $I_{\alpha, \mathfrak{m}_\alpha}$ is given precisely as $\mathbb{C}\left [h_0^k t_\alpha^l  \right ]$ for $l\geq \min\left \{ k, q\od \right \}$ in the proof of \cite{A2012}*{Lemma~3.11}.
Following \cite{A2012}, we set
\begin{equation}\label{eqn:Jalpha}
J_\alpha := \mathbb{C} [h_0^k t_\alpha^l  ] \mathfrak{S ( a^{ \perp}_\alpha )}, \textup{ with } k\in \mathbb{N}, l\geq \min\left \{ k, q\od \right \}.
\end{equation}
Here $q\od := \frac{1}{2}\dim (\gk_\alpha)\od$ and is 1 by the root data above.
Then $l\geq \min\{k, 1\}$ implies $J_\alpha =\left( \mathbb{C}\oplus t_\alpha\mathfrak{S}(h_0, t_\alpha)\right)\mathfrak{S}(\ak_\alpha^\perp) = \mathfrak{S}(\ak_\alpha^\perp) \oplus t_\alpha\mathfrak{S}(\ak)$. 
By definition of $t_\alpha$, we have $\mathbb{C}t_\alpha^\perp = \ker \alpha$ and $\mathbb{C}t_\alpha^\perp \supseteq \ak_\alpha^\perp$.
Therefore, we obtain the following $h_0$-independent formulation
\begin{equation} \label{eqn:JalphaNoH0}
J_\alpha = \mathfrak{S}(\ker\alpha)+t_\alpha\mathfrak{S}(\ak).
\end{equation}

Set $i=1, \dots, p$ and $j = 1, \dots, q$. In Definition~\ref{defn:noshiftBCSym}, we set $V = \ak^*$, $e_i := \alpha\bos_i$, $d_j := \alpha\fer_j$ as in $\Sigma := \Sigma(\gk, \ak)$. For $\Sigma$, the Weyl group $W_0$ is of Type \textit{BC}, from which follows Condition (i). 
Using (\ref{eqn:akForm}) and the symmetry in Condition (i), Condition (ii) can be written as
\begin{enumerate}
    \item[(ii')] $f(X+\alpha)=f(X)$ if $(X, \alpha) = 0$ for $\alpha = \pm\alpha_i\bos \pm \alpha_j\fer \in \iso{\Sigma}$.
\end{enumerate}

We identify the symmetric algebra $\Sk(\ak)$ with the polynomial algebra $\Pk(\ak^*)$ so that $\prod a_i \in \mathfrak{S}(\ak)$ corresponds to $p\in \mathfrak{P}(\ak^*)$:
\begin{equation} \label{eqn:identification}
    \mu \mapsto p(\mu) := \prod \mu(a_i), \; \text{for any $\mu\in\ak^*$.}
\end{equation}
%We provide a suitable formulation of \cite{A2012}*{Lemma~3.11} to show Condition (ii').
We set $\tau_\alpha$ by $\tau_\alpha(p)(X) = p(X+\alpha)$ for any function $p$ defined on the hyperplane 
\[
\Hyp_\alpha = \left \{ \mu \in \ak^*: \left ( \alpha , \mu \right ) = 0 \right \} \subseteq \ak^*.
\]
Then Condition (ii') can be rephrased as
\[
f \vert_{\Hyp_\alpha}= \tau_\alpha(f\vert_{\Hyp_\alpha}), \text{ for } \alpha  \in \iso{\Sigma}.
\]

\begin{lemma} \label{lem:transInv}
    Let $V$ be a vector space with basis $\{e_i\}_{i=1}^N$ and standard coordinates $\{x_i\}_{i=1}^N$. If $g \in\Pk(V) = \C[x_i]_{i=1}^n$ satisfies $g(X+e_1)=g(X)$, then $g\in \C[x_2, \dots, x_N]$. That is, $g\in \mathfrak{S}(\ker E_1)$ where $E_1:e_i\mapsto \delta_{1i}$ is the evaluation at $e_1$.
\end{lemma}
\begin{proof}
Suppose on the contrary $g= \sum_{i=1}^d c_i x_1^{k_i}m_i$ for $k_i \geq 1$, $c_i\neq 0$, and monomials $m_i\in \C[x_2, \dots, x_N]$. Then $g(x_1+1,\dots, x_N)-g(x_1, \dots, x_N)=0$ if and only if
\[
\sum p_k[(x_1+1)^{k}-x_1^{k}]= 0
\]
where $p_k = \sum c_{ik}m_{ik}\in \C[x_2, \dots, x_N]$ with distinct $m_{ik}$ for each $k\geq 1$. Since distinct $kx_1^{k-1}$ are linearly independent over the ring $\C[x_2, \dots, x_N]$, so are $(x_1+1)^{k}-x_1^{k}$ as their leading terms are $kx_1^{k-1}$. Hence all $p_k$ must be 0, and we have $c_i = 0$ for all $i$.
\end{proof}

The following proposition is also independently proved in \cite{SZ23}.

\begin{proposition} \label{prop:imGammaSupSym}
Let $\Sigma = \Sigma(\gk, \ak)$ and $\hcHomoGK: \mathfrak{U}^\kk(\gk) \rightarrow \mathfrak{S}(\ak)\cong \mathfrak{P}(\ak^*)$ be as above. We have $\Ima \hcHomoGK = \RingEv(\ak^*)$.
\end{proposition}
\begin{proof}
We check Conditions (i, ii') for $f\in J(\ak)$ (see (\ref{eqn:Ja}) and (\ref{eqn:JalphaNoH0})).
Condition (i) is equivalent to $f\in \mathfrak{S\left(a\right)}^{W_0}$ as $W_0$ is the Weyl group of Type $BC$.
As for Condition (ii'), it suffices to show
\begin{equation} \label{eqn:tauAlphaEqn}
    f \in J_\alpha = \mathfrak{S}(\ker\alpha)+t_\alpha\mathfrak{S}(\ak) \iff f\vert_{\Hyp_\alpha} \left(\mu\right) = f\vert_{\Hyp_\alpha}  \left(\mu + \alpha\right) ,\quad  \mu\in \Hyp_\alpha \subseteq \ak^*.
\end{equation}

Since the pairing between $\Hyp_\alpha$ and $\mathbb{C}t_\alpha$ is 0, it descends to a non-degenerate pairing between $\Hyp_\alpha$ and $\ak/\mathbb{C}t_\alpha$. Hence we can identify $\mathfrak{P}(\Hyp_\alpha)$ with $\mathfrak{S}(\ak/\mathbb{C}t_\alpha)$.
Since $t_\alpha \in \ker \alpha$, $\alpha$ descends to $\Tilde{\alpha}:\ak/\mathbb{C}t_\alpha \rightarrow \mathbb{C}$ defined by $\Tilde{\alpha}(A+\mathbb{C}t_\alpha) := \alpha(A)$.
Let $g \in \mathfrak{S}(\ak/\mathbb{C}t_\alpha) \cong \mathfrak{P}(\Hyp_\alpha)$. Suppose $g(\mu + \alpha) = g(\mu)$ for $\mu\in \Hyp_\alpha$. Then $g \in \mathfrak{S}(\ker \Tilde{\alpha})$ by Lemma~\ref{lem:transInv}. Therefore, 
\[
f\vert_{\Hyp_\alpha} \left(\mu\right) = f\vert_{\Hyp_\alpha}  \left(\mu + \alpha\right)  \iff f\vert_{\Hyp_\alpha} \in \mathfrak{S}(\ker \Tilde{\alpha}).
\]

Let $V\supseteq W$ be two finite dimensional vector spaces. Then the standard projection $V \rightarrow V/W$ induces a projection $\mathfrak{S}(V) \rightarrow \mathfrak{S}(V/W)$ and the exact sequence
\[
0 \rightarrow W\mathfrak{S}(V) \rightarrow \mathfrak{S}(V) \rightarrow \mathfrak{S}(V/W) \rightarrow 0.
\]
Now take $V = \ak$ and $W = \mathbb{C}t_\alpha$. Clearly $\ker \Tilde{\alpha} = \ker\alpha /\mathbb{C}t_\alpha$.
We see that $f\vert_{\Hyp_\alpha} \in \mathfrak{S}(\ker \Tilde{\alpha})$ must correspond to $g+h$ with $g\in \mathfrak{S}(\ker\alpha) \subseteq \mathfrak{S}(\ak)$ and $h\in t_\alpha\mathfrak{S}(\ak)$. That is,
\[
f\vert_{\Hyp_\alpha} \in \mathfrak{S}(\ker \Tilde{\alpha}) \iff f\in J_\alpha,
\]
proving (\ref{eqn:tauAlphaEqn}).
\end{proof}

\section{Spherical Representations}\label{sphRep}
We study finite dimensional irreducible highest weight modules of $\gk = \gl(2p|2q)$ which are $\kk$-spherical for $\kk=\gl(p|q)\oplus \gl(p|q)$. In this section, we discuss generalities, specialize the criteria in \cite{a2015spherical}, and prove Theorem~\ref{main:11Sph} in the end.

\subsection{Generalities}
We first present certain general statements regarding spherical representations.
Let $\Gk$ be a Lie (super)algebra admitting the Iwasawa decomposition $\Gk = \Kk\oplus \Ak \oplus \Nk$. Here $\Kk$ is the subalgebra fixed by the Cartan involution $\theta$, $\Ak \subseteq \Pk\ev$ is maximal toral where $\Pk$ is the $(-1)$-eigenspace of $\theta$, and $\Nk$ is a nilpotent subalgebra. 
Recall that for a $\Gk$-module $V$, if $V^\Kk := \{v\in V: X.v =0 \textup{ for all }X\in \Kk \} \neq \{0\}$, then $V$ is called ($\Kk$-)spherical. 
We denote $\Uk(\Gk)$ as $\Uk$, the $\Kk$-centralizer of $\Uk$ as $\Uk^\Kk$, and the set of weights of a module $V$ as $\supp{V}$.

\begin{proposition}\label{prop:uniSph}
If $V$ is a finite dimensional and irreducible $\Gk$-module, then $\dim V^{\Kk} \leq 1$. 
%That is, any $\Kk$-spherical vector is unique up to constant.
\end{proposition}

\begin{proof}
    Let $\Bk = \Hk \oplus \Nk$ be a Borel subalgebra of $\Gk$ where $\Hk \supseteq \Ak$ is a $\theta$-stable Cartan subalgebra. Denote the $\Bk$-highest weight vector of $V$ by $v$. Suppose on the contrary that there are two linearly independent spherical $v_1$ and $v_2$ in $V^\Kk$. Then we may find a non-zero linear combination $w = c_1 v_1 + c_2 v_2$ when expressed as a combination of weight vectors, the coefficient of $v$ is 0. Then the submodule
    \[
    \Uk(\gk)w = \Uk(\Nk^-)\Uk(\Ak)\Uk(\Kk)w
    \]
    does not contain $v$ and is thus proper, contradicting the irreducibility of $V$.
\end{proof}

\begin{lemma}\label{lem:scalar}
If $V$ is a $\Gk$-module with $\dim V^\Kk = 1$, then $u\in \Uk^\Kk$ acts by a scalar on $V^\Kk$.
\end{lemma}
\begin{proof}
If $v\in V^\Kk$, then $X.(u.v) = (-1)^{|u||X|} u.(X.v) = 0$ for any homogeneous $X\in \Kk$ and $u\in \Uk^\Kk$. %Thus $u.v$ is also spherical, and hence a scalar multiple of $v$. %The uniqueness of the spherical vector (Proposition~\ref{prop:uniSph}) implies the lemma.
\end{proof}

A $\Gk$-module is said to be \textit{cospherical} if $V^*$ is spherical. Equivalently, there exists a non-zero functional on $V$ such that $\Kk$ acts by 0 contragrediently. We assume that $\Hk$ is an even Cartan subalgebra containing $\Ak$, and is stable under $\theta$ which defines the superpair $(\Gk, \Kk)$. Let $\Tk = \Hk \cap  \Kk$. With slight abuse of notation, $\theta$ also denotes the involution on $\Hk^*$ induced from $\theta$.

Let $U_\lambda$ denote a finite dimensional irreducible $\Gk$-module of highest weight $\lambda \in \Hk^*$ with respect to a Borel subalgebra $\Bk$ containing $\Hk$ and $\Nk$. Let $v \in U_\lambda$ be a non-zero highest weight vector.

\begin{proposition} \label{prop:non0pair}
    Suppose $U_\lambda$ is cospherical and $\kappa \in (U_\lambda^*)^\Kk$ is non-zero. Then the canonical pairing $\langle v, \kappa \rangle$ is non-zero.
\end{proposition}
\begin{proof}
We have $\Gk = \Kk + \Bk$ from the Iwasawa decomposition.
We prove the claim by contradiction. Suppose $\langle v, \kappa \rangle = 0$. Consider $\langle u.v, \kappa \rangle$ for any $u\in \Uk$. By the Poincar\'e--Birkhoff--Witt theorem, we write $u =  u_K u_B $ with $u_K \in \Uk(\Kk), u_B \in \Uk(\Bk)$. Note that $\Uk(\Bk)$ acts on $v$ via a character and so does $\Uk(\Kk)$ on $\kappa$.
Then
\begin{equation*}
    \langle u.v, \kappa \rangle = \langle u_K u_B .v, \kappa \rangle
    = c\langle u_K .v, \kappa \rangle 
\end{equation*}
%where $c$ is a scalar determined by $u_B$. Notice $u_K \in \Uk(\Kk) = \C\oplus \Kk\ev \Uk(\Kk)\oplus \Kk\od\Uk(\Kk)$. 
where $c$ is a scalar determined by $u_B$. Notice $u_K \in \Uk(\Kk) = \C\oplus \Kk \Uk(\Kk)$.
%For arbitrary $u\in\Uk(\Kk)$ and $K\in \Kk\ev$, we have $\langle Ku.v, \kappa \rangle = -\langle u.v, K.\kappa \rangle = 0$;
%and for $\xi \in \Kk\od$, we have $\langle \xi u.v, \kappa \rangle = -\langle u.v, \xi.\kappa\ev-\xi.\kappa\od \rangle = 0$ (note $\xi.\kappa_{\overline{i}}=0$). Therefore, we get $\langle u.v, \kappa \rangle = cd \langle v, \kappa \rangle$ for another scalar $d$ determined by $u_B$.
For any $u\in\Uk(\Kk)$ and $K\in \Kk$, we have $\langle Ku.v, \kappa \rangle =  0$. Therefore, we get $\langle u.v, \kappa \rangle = cd \langle v, \kappa \rangle$ for another scalar $d$ determined by $u_B$.
Thus, $\kappa$ vanishes on $\Uk.v = U_\lambda$ which implies $\kappa = 0$, a contradiction.
\end{proof}

\begin{lemma} \label{lem:thetaMu}
    Suppose $U_\lambda$ is cospherical. Then $\lambda|_\Tk = 0$. That is, $\theta \lambda = -\lambda$.
\end{lemma}
\begin{proof}
    Let $\kappa$ be cospherical on $U_\lambda$ and $v\in U_\lambda$ be highest. Let $t \in \Tk$. Then $\langle t.v, \kappa\rangle  = \lambda(t)\langle v, \kappa\rangle$. On the other hand, $\langle t.v, \kappa\rangle = -\langle v, t.\kappa\rangle = 0$ since $\Tk \subseteq \Kk$. By Proposition~\ref{prop:non0pair}, $\lambda(t) = 0$.
\end{proof}

\begin{proposition} \label{prop:Sph=coSph}
    Suppose $U_\lambda$ is cospherical, then $U_\lambda$ is also spherical. Thus sphericity and cosphericity are equivalent for $U_\lambda$.
\end{proposition}
\begin{proof}
    Let $U_\lambda^\theta$ be the $\Gk$-module with the same vector (super)space as $U_\lambda$ but the representation map be pre-composed by $\theta$. 
    Then \textit{as $\Kk$-modules}, $U_\lambda^\theta \cong U_\lambda$ since $\theta$ fixes $\Kk$. Thus $(U_\lambda^\theta)^\Kk \cong U_\lambda^\Kk$.
    We also see that $\supp{U_\lambda^*} = \{-\mu: \mu \in \supp{U_\lambda}\}$. By Lemma~\ref{lem:thetaMu}, we get $\theta \lambda = -\lambda$, and this implies $U_\lambda^* \cong U_\lambda^\theta$. Then as $(U_\lambda^*)^\Kk$ is non-zero, $(U_\lambda^\theta)^\Kk \cong U_\lambda^\Kk$ is also non-zero. %Dual of irrep is irrep. Irreps are completely determined by weights.
    We refer to \cite{ShermanThesis}*{Proposition~6.4.2} for a similar construction and argument. 
\end{proof}

\subsection{Finite Dimensionality} \label{subsec:fD}
Recall in the Cheng--Wang decomposition (Proposition~\ref{prop:goodCW}), each component $W_\lambda^*$ has the following highest weight with respect to the Borel subalgebra $\bk^{\opp}\oplus \bk^{\std}$
\begin{equation} \label{eqn:CWweights}
-\sum_{i = 1}^p \lambda_i \left(\epu_i^+-\epu_i^-\right)-\sum_{j = 1}^q \langle \lambda_j'-p \rangle (\dau_j^+-\dau_j^-)= -\sum_{i = 1}^p \lambda_i \gamma_i\bos - \sum_{j = 1}^q \langle \lambda_j'-p \rangle  \gamma_j\fer \in \tk^*.
\end{equation}
Let us denote the above weight as $-\lambda^\natural$ by a slight abuse of notation. 
Recall $\bk_\natural^- = \bk^{\opp}\oplus \bk^{\std}\oplus \pminus$ from (\ref{eqn:b-natural}).
We are interested in $\gk$-modules $V_\lambda := V(-\lambda^\natural, \bk_\natural^-)$ and study which of them are $\kk$-spherical. 
We show that $V_\lambda$ are always finite dimensional using the results in Section~\ref{prelim}. %This will be useful when we consider its dual in Section~\ref{pairArg}.
%{\color{red}
%\begin{Definition}
%We say $\lambda$ is of $\bk$ \emph{finite type} if the corresponding Verma module $M(\lambda)$ has a finite dimensional quotient. 
%\end{Definition}
%This is equivalent to $\dim V(\lambda) < \infty$ since $V(\lambda)$ is a quotient of $M(\lambda)$.}

%In what follows, we examine how odd reflections give rise to different Borel subalgebras when applied to $\bk^\natural$, and how they change the highest weight of a $\gk$-module with respect to these Borel subalgebras.  

We first introduce some notation. For $\gl(2p|2q)$, any highest weight in the form of
\[
\sum x_i^\pm\epu_i^\pm + \sum y_j^\pm \dau_j^\pm.
\]
are specified by (1) the coefficients $x_i^\pm, y_j^\pm$, and (2) the order of $\epu^\pm$ and $\dau^\pm$ when describing the Borel subalgebra. We write the above weight as
\[
(\dots, \overset{\bullet}{x_i^\pm}, \dots, \overset{\times}{y_j^\pm}, \dots)
\]
where the coefficients are in the order of the $\epu\dau$-chain and we put a dot $\bullet$ (respectively a cross $\times$) over a coefficient of an $\epu_i^\pm$ (respectively $\dau_j^\pm$). Odd reflections only swap a $\bullet$ with a $\times$.
As Borel subalgebras are conjugate under the usual (even) reflections, swapping a $\bullet$ with a $\bullet$ or a $\times$ with a $\times$ will not change the order of these coefficients.

In this notation, if we take $\alpha = \dau^\pm_j-\epu^\pm_i$ in Theorem~\ref{thm:oddRefWghts}, we then have the following rule.
Let $\bk^{(1)}$ and $\bk^{(2)}$ be adjacent. If the $\bk^{(1)}$-highest weight $\lambda^{(1)}$ is given by
$(\cdots, \overset{\times}{x} | \overset{\bullet}{y}, \cdots)$ where $x, y$ are the entries where the odd reflection is applied, then the $\bk^{(2)}$-highest weight $\lambda^{(2)}$ is given by
\begin{enumerate}
    \item $(\cdots, \overset{\bullet}{y} |  \overset{\times}{x}, \cdots)$ if $x=-y$ (i.e. $(\lambda, \alpha) = 0$), or
    \item $(\cdots, \overset{\bullet}{y+1} | \overset{\times}{x-1}, \cdots)$ if $x\neq -y$ (i.e. $(\lambda, \alpha) \neq 0$).
\end{enumerate}

By Theorem~\ref{thm:oddRefWghts}, an odd reflection applied on a highest weight produce another highest weight with respect to the new Borel subalgebra. Therefore, we have the following lemma.
%As odd reflections are defined in such way that the resulting tuple is highest with respect to the resulting Borel subalgebra (Theorem~\ref{thm:oddRefWghts}), the dominance relation is retained. We have the following lemma.

\begin{lemma} \label{lem:oddRefDom}
Let $(\overset{\times}{x}, \overset{\times}{y} |\overset{\bullet}{z})$ be with $x\geq y$.
If $(\overset{\bullet}{u} | \overset{\times}{v}, \overset{\times}{w})$ is the result of applying the two odd reflections that switch $z$ with $y$ and then $x$, then $v \geq w$.
\end{lemma}

\begin{theorem}\label{thm:CWWghtsFD}
For any $\lambda\in \mathscr{H}(p, q)$, $V_\lambda$ is finite dimensional. 
\end{theorem}
\begin{proof}
In view of Lemma~\ref{lem:BorelSeq}, we may find certain Borel subalgebras (via odd reflections) so that any \textit{even} simple root is simple in some Borel subalgebra. 
If each corresponding $\lambda^{\alpha}$ is dominant as in Theorem~\ref{thm:fDBorelWght}, then we have the desired result.

Recall the Borel subalgebra $\bk_\natural^-$ of $\gk$ corresponds to the following chain (\ref{eqn:b-natural})
\begin{equation*}
    [\epu_1^- \cdots \epu_p^- | \dau_1^- \cdots \dau_q^- \dau_q^+ \cdots \dau_1^+ | \epu_p^+  \cdots \epu_1^+],
\end{equation*}
and all the simple roots of $(\bk_\natural^-)\ev$ are $\epu_1^- -\epu_2^-, \dots, \epu_p^- -\epu_p^+, \epu_p^+ -\epu_{p-1}^+, \dots, \epu_2^+ -\epu_1^+$ and $\dau_1^- -\dau_2^-, \dots, \dau_q^- -\dau_q^+, \dau_q^+ -\dau_{q-1}^+, \dots, \dau_2^+ -\dau_1^+$. In the above chain, the only missing even simple root is $\epu_p^- -\epu_p^+$, which is present in the following chain: 
\begin{equation}\label{eqn:secondChain}
    [\epu_1^- \cdots \epu_p^- \epu_p^+ |\dau_1^- \cdots  \dau_1^+ | \epu_{p-1}^+  \cdots \epu_1^+].
\end{equation}
Let us denote the Borel subalgebra and the highest weight with respect to (\ref{eqn:secondChain}) as $\bk^\sharp$ and $\lambda^\sharp$ respectively. 
We may choose the form $(-, -)$ on $\tk^*$ to be induced from the supertrace form on $\tk$. Then the $\bk$-dominance condition in Section~\ref{prelim} just means the integral coefficients of adjacent $\epsilon^\pm$ (respectively $\dau^\pm$) are non-increasing. Note there are no odd and non-isotropic roots in this case. 

Clearly, the original $-\lambda^\natural$ is $\bk_\natural^-$-dominant. The question is now whether or not $\lambda^\sharp$ is $\bk^\sharp$-dominant. By Theorem~\ref{thm:oddRefWghts}, we can obtain $\lambda^\sharp$ by a sequence of odd reflections performed on:
\[
-\lambda^\natural = (\overset{\bullet}{\lambda_1}, \dots, \overset{\bullet}{\lambda_p} | \overset{\times}{\nu_1}, \dots \overset{\times}{\nu_q}, \overset{\times}{-\nu_q}, \dots, \overset{\times}{-\nu_1} | \overset{\bullet}{-\lambda_p}, \dots,\overset{\bullet}{-\lambda_1})
\]
where $\nu_j = \langle \lambda'_j-p  \rangle$ for $\lambda \in \mathscr{H}(p, q)$. Denote the number of strictly positive entries in $(\nu_1, \dots, \nu_q)$ as $l$. Then $\lambda_p \geq l$ as $\lambda_p \geq q \geq \lambda_{p+1} = l$.

Denote $\lambda^\sharp = (\overset{\bullet}{\lambda_1}, \dots, \overset{\bullet}{\lambda_p}, \overset{\bullet}{\tau_1} | \overset{\times}{\sigma_1}, \dots, \overset{\times}{\sigma_{2q}} | \overset{\bullet}{-\lambda_{p-1}}, \dots, \overset{\bullet}{-\lambda_1})$. By Lemma~\ref{lem:oddRefDom}, we know that $\sigma_i \geq \sigma_j$ for any $i>j$ by induction. Therefore, we only need to prove $\lambda_p \geq \tau_1$ for $\bk^\sharp$-dominance. 

\textit{Case (i)} Suppose $\lambda_p \geq q$, then the sequence of $2q$ odd reflections performed on $-\overset{\bullet}{\lambda_p}$ adds \textit{at most} $2q$ to $-\lambda_p$ and we have $\tau_1 \leq -\lambda_p+2q \leq q \leq \lambda_p$.

\textit{Case (ii)} Otherwise $q>\lambda_p \geq l$ and 0 occurs in $(\dots, | \overset{\times}{\nu_1}, \dots, \overset{\times}{-\nu_1}, | \dots)$, and
\[
-\lambda^\natural = (\dots,\overset{\bullet}{\lambda_p} | \overset{\times}{\nu_1}, \dots \overset{\times}{\nu_l}, \underbrace{\overset{\times}{0}, \dots, \overset{\times}{0}}_{q-l},\underbrace{\overset{\times}{0}, \dots, \overset{\times}{0}}_{q-l}, \overset{\times}{-\nu_l}, \dots, \overset{\times}{-\nu_1} | -\overset{\bullet}{\lambda_p}, \dots ).
\]
We claim that $\tau_1 = l$. To get ${\tau_1}$, we first push $-\overset{\bullet}{\lambda_p}$ through $\overset{\times}{-\nu_l}, \dots, \overset{\times}{-\nu_1}$. Since all of these numbers are negative, each of the $l$ odd reflections adds 1 to $\overset{\bullet}{-\lambda_p}$, after which we get
\[
(\dots, \overset{\bullet}{\lambda_p}| \overset{\times}{\nu_1}, \dots \overset{\times}{\nu_l}, \underbrace{\overset{\times}{0}, \dots, \overset{\times}{0}}_{q-l},\underbrace{\overset{\times}{0}, \dots, \overset{\times}{0}}_{q-l} | \overset{\bullet}{-\lambda_p +l} | \overset{\times}{-\nu_l-1},\dots, \overset{\times}{-\nu_1-1} | \overset{\bullet}{-\lambda_{p-1}}, \dots).
\]
Notice $-\lambda_p+l\leq 0$ by assumption. 
Now as we keep applying the odd reflections, $\overset{\bullet}{-\lambda_p +l}$ will encounter the $2(q-l)$ zeros. 
But $q-l > \lambda_p -l$ by the assumption. This means there are more than enough zeros to make $-\lambda_p+l$ zero, since each time we have $(\dots, \overset{\times}{0} | \overset{\bullet}{x}, \dots)$ with $x<0$, the odd reflection adds 1 to $x$. 
If $x$ becomes 0, then the odd reflections just flip $\bullet$ with $\times$ after which we have the following:
\[
(\dots, \overset{\times}{\nu_1}, \dots \overset{\times}{\nu_l} | \overset{\bullet}{0} | \overset{\times}{0}, \dots, \overset{\times}{-1}, \overset{\times}{-\nu_l-1},\dots).
\]
Finally, we apply $l$ more odd reflections to push $\overset{\bullet}{0}$ through the $\times$ sequence to the left. Each reflection adds 1 to it. This proves $\tau_1=l \leq \lambda_p$.

Therefore, we always have $\lambda_p \geq \tau_1$, and $\lambda^\sharp$ is $\bk^\sharp$-dominant. 
\end{proof}

\subsection{The High Enough Constraints}\label{subsec:HEC}
We now specialize what the high enough condition (see Section~\ref{prelim}) means in our case. 
%We will see that this puts considerable constraints on which $\lambda \in \mathscr{H}(p, q)$ give rise to a spherical representation.
By Proposition~\ref{prop:even}, the highest weight of $V_\lambda$ on $\hk$ with respect to $\bk^-$ (see (\ref{eqn:b-})) is precisely
\begin{equation}\label{eqn:CWweightsOnA}
    -2\lambda_\hk^\natural := -\sum_{i = 1}^p 2\lambda_i \alpha_i\bos - \sum_{j = 1}^q 2\langle \lambda_j'-p \rangle  \alpha_j\fer \in \hk^*
\end{equation}
which vanishes on $\tk_+$. By slight abuse of notation, we denote its restriction on $\ak$ as
\[
-2{\lambda^\natural} := -2\lambda_\hk^\natural |_\ak .
\]

\begin{proposition} \label{prop:AlldridgePart}
Let $\lambda \in \mathscr{H}(p, q)$ and $V_\lambda$ be as above. If $\lambda_p > \langle \lambda_1'-p \rangle$,
then $V_\lambda$ is spherical.
\end{proposition}
\begin{proof}
Let us apply the sphericity condition given in Theorem~\ref{thm:AlldridgeSph} which prescribes the following:
\begin{enumerate}
    \item $\frac{ ( -2{\lambda^\natural}, \alpha  )}{\left ( \alpha, \alpha \right )} \in \mathbb{N}$ where $\alpha \in \Sigma^-\ev(\gk, \ak)$ (note that we picked $\bk^-$ here),
    \item $ ( -2{\lambda^\natural}, \beta  )>0$ for any isotropic $\beta \in \Sigma^-(\gk, \ak)$.
\end{enumerate}

Note that $-2\lambda_\hk^\natural|_{\tk_+} = 0$ as $\tk_+ = \hk \cap \kk$. The root data in Section~\ref{supPairs} shows:
$$\Sigma^-(\gk, \ak) = \left \{ -\alpha_j\bos-\alpha_k\bos, -\alpha_j\fer-\alpha_k\fer\right \}\sqcup \left \{-\alpha_j\bos+\alpha_k\bos, -\alpha_j\fer+\alpha_k\fer: j<k\right \} \sqcup \left\{-\alpha_j\bos\mp \alpha_k\fer \right \}.$$
%and the only isotropic positive roots are $\alpha_j\bos\pm\alpha_k\fer$. 
%Also, we have $\left ( \alpha_j\bos, \alpha_k\bos \right ) = -\left ( \alpha_j\fer, \alpha_k\fer \right ) = \delta_{jk}$, $\left ( \alpha_j\bos, \alpha_k\fer \right ) = 0$.

Take $\alpha = -2\alpha_j\bos, -2\alpha_j\fer, -\alpha_j\bnf\mp \alpha_k\bnf$ respectively. Then Condition (1) gives
\begin{align*}
    \frac{ ( -2{\lambda^\natural},- 2\alpha_j\bos  )}{4} &= \lambda_j \geq 0, \; &&\frac{ ( -2{\lambda^\natural}, -\alpha_j\bos \mp \alpha_k\bos  )}{2} = \lambda_j\pm \lambda_k \geq 0, \\
    \frac{ ( -2{\lambda^\natural}, -2\alpha_j\fer  )}{-4} &= \langle \lambda_j'-p \rangle \geq 0 , \; &&\frac{ (-2{\lambda^\natural}, -\alpha_j\fer \mp \alpha_k\fer  )}{-2} =  \langle \lambda_j'-p \rangle \pm \langle \lambda_k'-p \rangle\geq 0,
\end{align*}
by (\ref{eqn:akForm}) and (\ref{eqn:CWweightsOnA}), and all of them are vacuously true for $\lambda \in \mathscr{H}(p, q)$.

As for Condition (2), we write
\[
 ( -2{\lambda^\natural}, -\alpha_j\bos\mp \alpha_k\fer  ) = 2\lambda_j \mp 2 \langle \lambda_k'-p  \rangle >0
\]
for any $j, k$ which implies $\lambda_1 \geq \cdots \geq \lambda_p > \langle \lambda_1'-p \rangle \geq \cdots \geq  \langle \lambda_q'-p  \rangle \geq 0$.
\end{proof}
%In view of {\color{red} natural/chi map in SV, the image of these partitions is not even Zariski dense.}

Note this is only a \textit{sufficient} condition for sphericity.
In particular, we see that \textit{such a partition must have a $p$-th part}. Clearly, the trivial representation, albeit spherical, is not high enough. We would like to point out that the above condition is insufficient to show the vanishing properties for proving Theorem~\ref{main:RESULT}. See Remark \ref{rmk:whyInsufficient} for an example.
%we need ``enough" spherical representation in the form of $V(\lambda^\natural)$ to show the vanishing properties, in order to give a proof of Theorem~\ref{main:RESULT}. 
This means a new approach is needed.  

\subsection{Sphericity and Quasi-sphericity} \label{subsec:sphQuasi}
The rest of this section is dedicated to the study of sphericity of $V_\lambda$ when $p=q=1$ (Theorem~\ref{main:11Sph}) so we set $\gk = \gl(2|2)$ and $\kk = \gl(1|1)\oplus \gl(1|1)$. Let us assume that $\lambda \neq \varnothing$. Denote the standard characters of the diagonal Cartan subalgebra $\tk$ as $\epu^+, \epu^-, \dau^+, \dau^-$ from top left to bottom right. 
We use the following matrix to show our choice of coordinates:
\[
\begin{pmatrix}
\ast & X_1 & \eta_{11} & \eta_{12}\\ 
Y_1 & \ast & \eta_{21} & \eta_{22}\\ 
\xi_{11} & \xi_{12}  & \ast & X_2\\ 
\xi_{21}  & \xi_{22}  & Y_2 & \ast
\end{pmatrix}.
\]
Thus, a letter appearing above represents the matrix with 1 in the corresponding entry and 0 elsewhere. 
We abbreviate $\diag(a, b, c, d)$ as $\dangle{abcd}$. Note $\kk$ is spanned by the diagonal Cartan subalgebra together with $\eta_{ii}, \xi_{ii}$ for $i=1, 2$.
In what follows, $\gk_1$ (respectively $\gk_{-1}$) is the subalgebra spanned by $\eta_{ij}$ (respectively $\gk_{-1}$), and $\gk = \gk_{-1}\oplus \gk\ev \oplus \gk_1$.
%Let $V(\mu, \mathfrak{b})$ denote the irreducible $\gk$-module with $\mathfrak{b}$-highest weight $\mu$. We may write $\bk$ as its unique $\epsilon\delta$-chain.

For a non-empty (1, 1)-hook partition $(a, 1^b)$, the weight of $V_\lambda=V(-\lambda^\natural, \bk_\natural^-)$ we investigate is 
\[
-\lambda^\natural = (a | b,-b |-a) = a(\epu^- -\epu^+)+b(\dau^- -\dau^+)
\]
with respect to $\bk_\natural^-$, or the chain $[\epu^- |\dau^- \dau^+ | \epu^+]$.
We apply two odd reflections (swapping $\dau^+$ with $\epu^+$ and then $\dau^-$ with $\epu^+$) to get $[\epu^-  \epu^+|\dau^- \dau^+]$ which is conjugate to the standard Borel subalgebra $\bk^{\std}$ (corresponding to $[\epu^+  \epu^-|\dau^+ \dau^-]$). By computing the odd reflections, we get the new highest weight
%After changing to the Kac Borel subalgebra $\Breve{\bk} = \bk^\sharp$ (corresponding to $[\epu^+ \epu^- \dau^+ \dau^- ]$), we get (see Subsection~\ref{subsec:fD})
\begin{align} \label{eqn:2cases}
\Breve{\lambda} &=\begin{cases}
(a, -a+2 | b-1, -b-1)  & \text{ if } b\neq a-1 \\
(a, -b | b, -a) & \text{ if } b = a-1 
\end{cases} \\
&=\begin{cases}
a\epu^++(-a+2)\epu^-+(b-1)\dau^++(-b-1)\dau^- & \text{ if } b\neq a-1 \\
a\epu^++(-b)\epu^-+b\dau^++(-a)\dau^- & \text{ if } b = a-1 
\end{cases} .
\end{align}
Thus $V_\lambda = V(\Breve{\lambda}, \bk^{\std})$. 
The advantage of using $\bk^{\std}$ (which is \textit{distinguished}) is that it allows us to construct the finite dimensional \textit{Kac module} $K(\Breve{\lambda} )$ using an irreducible highest weight $\gk\ev$-module $\Breve{W}(\Breve{\lambda})$ (here $\gk\ev = \gl(2)\oplus\gl(2)$). Specifically, we extend the action of $\gk\ev$ trivially to $\mathfrak{r} = \gk\ev \oplus \gk_1$, and define
\[
K(\Breve{\lambda}) := \Ind_{\mathfrak{r}}^\gk \Breve{W}(\Breve{\lambda}).
\]
Then $V_\lambda$ is its irreducible quotient. 
As a vector space, $K(\Breve{\lambda}) $ is $\bigwedge (\gk_{-1}) \otimes \Breve{W}(\Breve{\lambda})$. When $\lambda$ is clear from the context, we will just write $\Breve{W}$ instead $\Breve{W}(\Breve{\lambda})$.%This is a special kind of parabolic induction. 

Let us begin with the simplest example where $\lambda$ is given by a single box $(1) = \square$. This is the second case in (\ref{eqn:2cases}) so the Kac highest weight is $(1, 0|0, -1)$. %, identical to the Cheng--Wang weight in appearance. 
In other words,
$V((1 | 0,0 | -1), \bk_\natural^-) = V((1,0 | 0,-1), \bk^{\std})$,
and $V := V((1|0,0|-1), \bk_\natural^-)$ appears as a quotient of $K := K(1,0|0,-1)$.% which has dimension $16\cdot 4 = 64$ since $\Breve{W}$ is a 4 dimensional $\gk\ev$-module. We have the following quick observation.

\begin{proposition}
%Let the notation be as above.
%Let $\lambda = \square \in \mathscr{H}(1, 1)$. 
Following the above notation, $V$ is spherical, while $K$ is not.
\end{proposition}
\begin{proof}
The claim that $V$ is spherical follows from Proposition~\ref{prop:AlldridgePart}.
A more direct way to see this is to observe that $V \cong \mathfrak{psl}(2|2) = \mathfrak{sl}(2|2)/\mathbb{C}I_4$ since $\mathfrak{psl}(2|2)$ is simple as a Lie superalgebra and as a $\gk$-module. Its highest weight is $\epu^+-\dau^- = (1, 0|0, -1)$ as a quotient of the adjoint representation of $\gk$. 
Then $\kk$ acts trivially on the equivalent class of $\omega = \dangle{1 (-1) 1(-1)}$ in $\mathfrak{psl}(2|2)$ since $\eta_{11}$ and $\eta_{22}$ have weights $(1, 0 | -1, 0)$ and $(0, 1 | 0, -1)$ respectively under the adjoint action; the two $\xi_{ii}$ have opposite weights. All the basis vectors in $\kk$ act as 0 on $\dangle{1(-1) 1 (-1)}$.

We show that $K$ is not spherical by way of contradiction. Suppose a spherical vector $\omega = \sum c_i \Xi_i\otimes w_i$ exists, where $\Xi_i \in \bigwedge, w_i\in \Breve{W}$, and $c_i$ are some constants. Then $\omega$ should be annihilated by $\xi_{11}$ and $\xi_{22}$ and have weight 0. Then each $\Xi_i$ must contain both $\xi_{11}$ and $\xi_{22}$ and hence its weight contains $-\epu^+ - \epu^- + \dau^+ +\dau^-$. In order to have weight 0 for each $\Xi_i\otimes w_i$, the weight of $w_i$ must contain $\epu^++\epu^- -\dau^+ -\dau^-$. However, such weights do not appear in $\Breve{W}(1, 0|0, -1)$. 
\end{proof}

The contrast between the two modules reveals an important fact: the Kac module $K$ must have a non-zero vector that descends to its irreducible quotient $V$ such that the image is spherical. 
This motivates the following definition of a \textit{quasi-spherical vector}.

\begin{Definition}
Let $\Gk$ and $\Kk$ be two Lie (super)algebras such that $\Kk\subseteq \Gk$. 
Let $U$ be a $\Gk$-module with a unique maximal submodule $M$. A non-zero vector $v\in U$ is said to be $\Kk$-\textit{quasi-spherical} if $\Kk. v \subseteq M$ and $\Uk(\Gk).v = U$, and $U$ is said to be $\Kk$-\textit{quasi-spherical}.
\end{Definition}

We suppress ``$\Kk$-" if it is clear from the context. 
If $U$ is a highest weight module, then $U$ has unique maximal submodule. If $v\in U$ is cyclic ($\Uk.v = U$), then it descends to the unique irreducible quotient of $U$. If we want to show that $u\in U$ is in a proper submodule, it suffices to prove that $\Uk(\Gk).u$ does not contain any cyclic vectors. Note that a quasi-spherical vector need not be cyclic. 

%When $p=q=1$, there are only two cases (\ref{eqn:2cases}). 
We first prove that for typical $\Breve{\lambda}$ ($b\neq a-1$ in (\ref{eqn:2cases})), $K(\Breve{\lambda})$ is spherical (so trivially quasi-spherical).
%We show that $K(\Breve{\lambda})$ is always $\kk$-quasi-spherical.

%In what follows, we fix a choice a weight vector in $\Breve{W}$ by its weight and we specify the action of $X_i$ by $X_1.(m, n, k, l)=(m+1, n-1, k, l)$, $X_2.(m, n, k, l)=(m, n, k+1, l-1)$ provided that they appear as weights. Then the $Y_i$'s act accordingly. This is well defined as all the weight spaces are 1 dimensional. We record specific computations in Appendix \ref{app:11brackets}. 
%Note that using our notation, $\dangle{x_1, x_2, x_3, x_4}.1\otimes (y_1, y_2, y_3, y_4)$ is $\sum_{i=1}^4x_iy_i 1\otimes (y_1, y_2, y_3, y_4)$

\begin{theorem}\label{thm:b!=a-1}
Let $\lambda$ be $(a, 1^b)  \in \mathscr{H}(1, 1)$ with $b\neq a-1$. Let $v\in \Breve{W}(\Breve{\lambda})$ be a non-zero vector of weight $(1, 1| -1, -1)$. Then $\omega := \xi_{11}\xi_{22}\otimes v \in K(\Breve{\lambda})$ is spherical in $K(\Breve{\lambda}) = V_\lambda$.
\end{theorem} 
\begin{proof}
 If $b\neq a-1$, then $\Breve{\lambda} = (a, -a+2 | b-1, -b-1)$. 
Then the element $Y_1$ (respectively $Y_2$) lowers the first (respectively the third) entry by 1 all the way down to $-a+2$ (respectively $-b-1$), while raises the second (respectively the fourth) entry by 1 all the way up to $a$ (respectively $b-1$).   
Thus the weight $(1, 1| -1, -1)$ occurs along the way. 
We show that $\omega = \xi_{11}\xi_{22}\otimes v$ is spherical, that is, the vector $\omega$ (\textit{a}) has weight $(0, 0| 0, 0)$, (\textit{b}) is annihilated by $\xi_{ii}$, and (\textit{c}) is annihilated by $\eta_{ii}$, for $i=1, 2$. Then (\textit{a}) follows from the fact that $\xi_{11}\xi_{22}$ has weight $(-1, -1| 1, 1)$ and (\textit{b}) follows from $\xi_{ii}^2=0$. 

For (\textit{c}), we use the following computations in $\Uk$. Since $\eta_{ij}$ acts on $1\otimes\Breve{W}$ by 0, we use $A\equiv B$ for $[A]=[B]$ in $\Uk/\left(\Uk\gk_1\right)$ as equivalent classes. Thus, if $A\equiv B$, the actions of $A, B$ on $1\otimes \Breve{W}$ are the same. The explicit computations below follow from the superbracket table in Appendix \ref{app:11brackets}:
\begin{align}\label{eqn:eta11}
        \eta_{11}\xi_{11}\xi_{22} &= ([\eta_{11}, \xi_{11}]-\xi_{11}\eta_{11})\xi_{22} \notag\\
        & =  \dangle{1010}\xi_{22}+\xi_{11}\xi_{22}\eta_{11} \notag \\
        &\equiv \xi_{22}\dangle{1010},\\
\label{eqn:eta22}
\eta_{22}\xi_{11}\xi_{22} &= -\xi_{11}\eta_{22}\xi_{22} \notag \\
        & =  -\xi_{11}([\eta_{22}, \xi_{22}]-\xi_{22}\eta_{22})\notag \\ &\equiv -\xi_{11}\dangle{0101} .
\end{align}

Since $v$ has weight $(1, 1|-1, -1)$, $\dangle{1010}.1\otimes v= 0$, and by (\ref{eqn:eta11}), $\eta_{11}$ acts on $\omega$ by 0. Similarly, by (\ref{eqn:eta22}), $\eta_{22}$ acts on $\omega$ as 0. Therefore $\omega$ is spherical.

Note that the Weyl vector $\Breve{\rho}$ has coordinate $\frac{1}{2}(-1, -3|3, 1)$, and direct computations show that $(\Breve{\lambda}+\Breve{\rho}, \alpha) \neq 0$ for isotropic $\alpha$ in $\{ (1, 0|-1, 0), (1, 0|0, -1), (0, 1|-1, 0), (0, 1| 0, -1)\}$.
Thus $\Breve{\lambda}$ is typical, and $K(\Breve{\lambda}) = V_\lambda$ is irreducible and spherical.
%A quick $\sltwo$ computation shows (see Appendix \ref{app:11brackets}) that $Y_2X_2-Y_1X_1$ in (\ref{eqn:eta2112}) acts on $1\otimes v$ by the scalar $(b+a)(b-a+1)$ which is non-zero by the assumption that $b\neq a-1$. 
%So $\eta_{21}\eta_{12}.\omega$ is a non-zero multiple of $1\otimes v \in 1\otimes \Breve{W}(\Breve{\lambda})$.
%Therefore $\omega$ is cyclic and descends to a non-trivial spherical vector in $V_\lambda$.
\end{proof}

The following discussion is needed for the second case.
Suppose $\Lk = \Lk_{-1} \oplus\Lk\ev \oplus \Lk_{1}$ is a Type I Lie superalgebra. Let $\left\{\xi_i\right\}$, $\left\{X_j\right\}$, $\left\{\eta_k\right\}$ be bases for the three summands respectively. We recall that the Poincar\'e--Birkhoff--Witt theorem says the following set
\begin{equation} \label{eqn:PBWbasis}
    \left\{\xi_{i_1}\cdots \xi_{i_m} X_{j_1}\cdots X_{j_l}\eta_{k_1}\cdots \eta_{k_n}: i_1 < \cdots <i_m, j_1\leq \cdots \leq j_l, k_1<\cdots <k_n\right\}
\end{equation}
is a basis for $\Uk(\Lk)$, called the PBW basis.
Let $\deg \xi_i = -1, \deg X_j = 0, \deg \eta_k = 1$. Let $\Uk^d(\Lk)$ be the span of PBW basis vectors with $n-m = d$ as in (\ref{eqn:PBWbasis}), so
\[
\Uk(\Lk) = \bigoplus_{d\in \mathbb{Z}} \Uk^d(\Lk), \; \textup{with } \Uk^d(\Lk)\Uk^f(\Lk) \subseteq \Uk^{d+f}(\Lk).
\]
This grading is well-defined since the Lie superbracket respects the short grading on $\Lk$, and the multiplication in $\Uk(\Lk)$ respects the Lie superbracket.

Now we let $W$ be a $\Lk\ev$-module and the action is extended to $\Lk_1$ trivially. Consider $K(W) := \Ind_{\Lk\ev\oplus \Lk_1}^\Lk W$. Then $K(W) \cong \Uk(\Lk) \otimes W = \bigwedge (\Lk_{-1})\otimes W$
as vector spaces. Moreover, $K(W)$ inherits a grading from $\Uk(\Lk)$ in the sense that 
\begin{equation} \label{eqn:actionDegree}
    \Uk^d(\Lk)\bigwedge \nolimits^{k} (\Lk_{-1})\otimes W = \bigwedge \nolimits^{k-d} (\Lk_{-1})\otimes W.
\end{equation}
In particular, (\ref{eqn:actionDegree}) is valid for our Kac module $K(\Breve{\lambda})$.

\begin{theorem}\label{thm:b=a-1}
Let $\lambda = (a, 1^b)  \in \mathscr{H}(1, 1)$ with $b = a-1$. If $v\in \Breve{W}(\Breve{\lambda})$ is a non-zero vector of weight $(1, 0| -1, 0)$, then $\omega := \xi_{11}\otimes v $ is quasi-spherical in $K(\Breve{\lambda})$.
\end{theorem}

\begin{proof}
If $b=a-1$, then $\Breve{W}$ has highest weight $\Breve{\lambda} = (a, -a+1 | a-1, -a)$ and the weight $(1, 0| -1, 0)$ occurs in $\Breve{W}$, c.f. the first paragraph in the proof of Theorem~\ref{thm:b!=a-1}. To see $\omega$ is cyclic, we note that
\[
\eta_{12}\xi_{11} = X_2 -\xi_{11}\eta_{12} \equiv X_2
\]
so $\eta_{12}.\omega = 1\otimes X_2v$. Since $X_2v$ has weight $(1, 0 | 0, -1)$, which does appear for $a\geq 1$, $1\otimes X_2v$ is non-zero. By definition, $\Breve{W}$ is simple. Hence $\omega$ generates $K(\Breve{\lambda})$. %Hence we can raise it to the highest vector $1\otimes (a, -b, b, -a)$ in $K(\Breve{\lambda})$. 

To show that $\omega$ is quasi-spherical, we prove that $\kk.\omega$ lies in a submodule of $K(\Breve{\lambda})$, thus in the maximal submodule. This is equivalent to showing that no cyclic vectors exist in $\kk.\omega$.
Let us compute $\kk.\omega$. Since $\omega$ is a $(0, 0| 0, 0)$-weight vector, it suffices to consider the action of $\xi_{ii}$ and $\eta_{ii}$ for $i=1, 2$.
Direct computations in $\Uk$ show that $\eta_{11}\xi_{11} = \dangle{1010}-\xi_{11}\eta_{11}, \eta_{22}\xi_{11} = -\xi_{11}\eta_{22}$ and they both yield 0 on $1\otimes v$. So $\eta_{ii}.\omega=0$. Note $\xi_{11}.\omega = 0$ by construction. The only non-trivial result comes from $\xi_{22}$. By letting $\omega' := \xi_{11}\xi_{22}\otimes v$, we have
\(
\kk .\omega = \mathbb{C} \omega' .
\)
%Let $\omega' $. We show that $\omega'$ is not cyclic so $\kk.\omega$ is contained in the maximal submodule.

Suppose by contradiction that $\omega'$ is cyclic. Then there exists $u\in \Uk$ such that $u\omega'\in 1\otimes \Breve{W}$. 
As $\omega' =\xi_{11}\xi_{22}\otimes v \in \bigwedge\nolimits^{-2}(\gk_{-1})\otimes \Breve{W}$, we see that $u$ must be in $\Uk^2$ for $u\omega' \in 1\otimes \Breve{W} = \bigwedge\nolimits^0(\gk_{-1})\otimes \Breve{W}$ (\ref{eqn:actionDegree}). 
This means $u$ must be a linear combination of degree 2 PBW basis vectors in $\Uk^2$. In particular, each basis vector has at least 2 $\eta_{ij}$'s. 
When $p=q=1$, $\gk_{-1}$ is spanned by $\eta_{11}, \eta_{12}, \eta_{21}, \eta_{22}$ and there are $\binom{4}{2}=6$ such combinations. 
It turns out that we need (\ref{eqn:eta11}) and (\ref{eqn:eta22}) as above.

In (\ref{eqn:eta11}), note $\xi_{22}$ has weight $(0, -1| 0, 1)$ so $\xi_{22}\otimes v$ has weight $(1, -1| -1, 1)$. Then $\dangle{1010}\xi_{22} \otimes v= 0$. Hence $u\eta_{11}.\omega'=0$ for any $u\in \Uk$.

In (\ref{eqn:eta22}), the Cartan subalgebra element $\dangle{0101}$ acts on $1\otimes v$ by 0, and $\eta_{22}.\omega'=0$. Hence $u\eta_{22}.\omega'=0$ for any $u\in \Uk$.

In addition, we have
\begin{align}
    \eta_{12}\xi_{11}\xi_{22} &= (X_2-\xi_{11}\eta_{12})\xi_{22}\notag \\ 
    &= X_2\xi_{22} -
    \xi_{11}(X_1-\xi_{22}\eta_{12}) \notag\\ 
    &\equiv \xi_{12}+\xi_{22}X_{2}- \xi_{11}X_1 \notag,
\end{align}
    from which we get 
\begin{align}\label{eqn:eta2112}
    \eta_{21}\eta_{12}\xi_{11}\xi_{22} &\equiv  \eta_{21}\xi_{12}+\eta_{21}\xi_{22}X_{2}- \eta_{21}\xi_{11}X_1 \notag\\
    &= \dangle{0110} -
    \xi_{12}\eta_{21}+ (Y_2-\xi_{22}\eta_{21})X_2 - (Y_1-\xi_{11}\eta_{21})X_1 \notag \\ 
    &= \dangle{0110} -
    \xi_{12}\eta_{21}+ Y_2X_2-\xi_{22}(\eta_{22}-X_2\eta_{21}) \notag \\
    &\quad - Y_1X_1+\xi_{11}(-\eta_{11}-X_1\eta_{21}) \notag \\ 
    &\equiv \dangle{0110}+ Y_2X_2- Y_1X_1.
\end{align}

The situation in (\ref{eqn:eta2112}) requires some $\sltwo$ calculations. In (\ref{eqn:eta2112}), the first term gives $-1\otimes v$. 
The module $\Breve{W}(a, -b|b, -a)$ is $L(a, -a+1) \otimes L(a-1, -a)$ as an irreducible $\gl(2)\otimes \gl(2)$-module  (see Appendix \ref{app:11brackets}). Standard $\sltwo$ computations show that the second term gives $Y_2X_2v = a(2a-a)v = a^2v$. Similarly, the third term gives $-Y_1X_1v = -(a+1)(2a-a-1)v = (1-a^2)v$. All three terms sum up to 0. Hence $\eta_{21}\eta_{12}.\omega'=0$.

Any of the 6 combinations of two $\eta_{ij}$'s is either $\eta_{21}\eta_{12}$, or a combination that ends in $\eta_{11}$ or $\eta_{22}$. These computations imply that any $u\in \mathfrak{U}$ of degree 2 satisfies $u.\omega' = 0$. Hence $\omega'$ is not cyclic. 
\end{proof}

\begin{remark}
Any choice of quasi-spherical vector has the same image up to some constant in the quotient due to the uniqueness of spherical vectors (Proposition~\ref{prop:uniSph}). In particular, the argument still works if one chooses $\omega := \xi_{22}\otimes v'$ for a non-zero vector $v'$ of weight $(0, 1|0, -1)$. %This suggests that the two different choices descend to the same spherical vector in the quotient.
\end{remark}

\begin{remark}
By the discussion in Section~\ref{supPairs}, $V_\lambda$ is guaranteed to be spherical given $a = \lambda_1 > \left\langle \lambda_1'-1 \right\rangle = b$ which of course covers the case $b = a-1$ (missing ``half" of the ``$ab$-quadrant"). However, we believe it is beneficial to record a different approach to Theorem~\ref{thm:b=a-1} which requires minimal algebraic machinery and is unified with Theorem~\ref{thm:b!=a-1}.
\end{remark}

\begin{proof}[Proof of Theorem~\ref{main:11Sph}]
Let $\lambda = (a, 1^b)\in \mathscr{H}(1, 1)$. If $\lambda = \varnothing$, then $V_\lambda$ the trivial module which is spherical. Suppose $\lambda\neq \varnothing$. If $b\neq a-1$, Theorem~\ref{thm:b!=a-1} says that $V_\lambda$ has a spherical vector; if $b = a-1$, Theorem~\ref{thm:b=a-1} says that $K(\Breve{\lambda})$ is quasi-spherical, indicating that $V_\lambda$ is spherical. %In both cases, $V(\lambda^\natural)$ is spherical.
\end{proof}

\section{Generalized Verma Modules and Proof of Theorem~\ref{main:RESULT}} \label{pairArg}
In this section, we present an algebraic proof of Theorem~\ref{main:RESULT} assuming Conjecture~\ref{conj:sph}. 
We construct the generalized Verma modules $M_{-\lambda}$ and study the cosphericity of them. We then utilize the grading on $M_{-\lambda}$ to show the ``vanishing action" of $D_\mu$ for suitable $\lambda$ which eventually proves Theorem~\ref{main:RESULT}.

\subsection{Generalized Verma Modules}
Recall that $W_\lambda^*$, the component of the Cheng--Wang decomposition, is a $\kk$-module with highest weight $-\lambda^\natural_\tk$ (see (\ref{eqn:-gammaWt})) with respect to $\bk^{\opp}\oplus \bk^{\std}$. We may extend the $\kk$-action on the finite dimensional irreducible $W_\lambda^*$ trivially to $\pminus$ and define
\begin{equation} \label{eqn:I-lambda}
    M_{-\lambda} := \Ind_{\kk+\pminus}^\gk W_\lambda^*.
\end{equation}
This is a parabolic induction and we call this the generalized Verma module. Note as $\kk$-modules, 
\begin{equation} \label{eqn:IDecomp}
    M_{-\lambda} \cong \Sk(\pplus)\otimes  W_\lambda^* \cong \bigoplus_{\mu \in \cH} W_\mu \otimes W_\lambda^*
\end{equation}
where the sum is taken over all the hook partitions. The upshot of this construction is threefold: 

\begin{enumerate}
    \item The irreducible quotient of $M_{-\lambda}$ is $V_\lambda = V(-\lambda^\natural, \bk_\natural^-) =  V(-2\lambda_\hk^\natural, \bk^-)$;
    \item $M_{-\lambda}$ has a unique cospherical functional $\phi$ in the restricted dual;
    \item This $\phi$ descends to $\kappa \in (V_\lambda^*)^\kk$ if $(V_\lambda^*)^\kk \neq \{0\}$.
\end{enumerate}
Point (1) is clear from the construction. Point (2) is Proposition~\ref{prop:cosphUnique}. Point (3), proved in Proposition~\ref{prop:kappaphi}, eventually leads to Theorem~\ref{main:RESULT} (see Theorem~\ref{thm:vanCon}). 
%allows us to show that the Shimura operators act by 0 on this functional. Point (3), explained in the next Subsection, helps to prove Theorem~\ref{main:11Sph}.

To study the (co)sphericity of $M_{-\lambda}$, we introduce a natural grading on $M_{-\lambda}$ compatible with the $\gk$-action as follows. Recall that
\[
\gk = \pplus \oplus \kk \oplus \pminus
\]
is a $(1, 0, -1)$-grading defined on $\gk$. Since the Lie superbracket respects this grading, it induces well-defined grading on $\Uk$ by setting $\deg \eta =  1$, $\deg X = 0$, and $\deg \xi = -1$ for $\eta \in \pplus$, $X \in \kk$, and $\xi \in \pminus$, c.f. (\ref{eqn:PBWbasis}) and the discussion there. By taking bases $\{\eta_i\}$, $\{X_j\}$, and $\{\xi_k\}$ for $\pplus$, $\kk$, and $\pminus$ respectively, we may set
\begin{align*}
    \Uk^d := \Spn \{ &\eta_{i_1}\cdots \eta_{i_m} X_{j_1}\cdots X_{j_l}\xi_{k_1}\cdots \xi_{k_n}:\\
    &i_1 < \cdots <i_m, j_1\leq \cdots \leq j_l, k_1<\cdots <k_n, m-n=d\} 
\end{align*}
on $\Uk$ so $\Uk = \bigoplus \Uk^d$. Let $|\lambda| = l$. Now on $M_{-\lambda}$, we define a grading such that 
\[
    M_{-\lambda}^{-l+k} := \Sk^k(\pplus) \otimes W_\lambda^*
\]
so $M_{-\lambda} = \bigoplus_{k\geq 0} M_{-\lambda}^{-l+k}$ (\ref{eqn:IDecomp}). We set $M_{-\lambda}^{n} = \{0\}$ for $n < -l$ so the lowest degree for which $M_{-\lambda}^n$ is non-zero is $n = -l$.  We have
\begin{equation} \label{eqn:degAct}
    \Uk^d.M_{-\lambda}^{-l+k} \subseteq M_{-\lambda}^{-l+k+d}.
\end{equation}

The full dual $M_{-\lambda}^*$ is way too large to get hold of. However, we may equip the \textit{restricted dual}
\begin{equation} \label{eqn:kFinDual}
    M_{-\lambda}^{\mathsf{fin}*} := \bigoplus (W_\mu \otimes W_\lambda^*)^* \cong \bigoplus W_\mu^* \otimes W_\lambda
\end{equation}
(c.f. (\ref{eqn:IDecomp})) with the contragredient $\gk$-action. 
Explicitly, for homogeneous $X\in \Uk^d$, $f\in W_\mu^* \otimes W_\lambda$, $v_\nu \in W_\nu\otimes W_\lambda^*$, $\langle v_\nu, X.f \rangle = \pm\langle X.v_\nu, f \rangle = 0$ for all but finite $\nu$'s, and thus $X.f \in M_{-\lambda}^{\mathsf{fin}*}$. 
Note every functional in $M_{-\lambda}^{\mathsf{fin}*}$ vanishes on all but finitely many summands $W_\mu\otimes W_\lambda^*$.

\begin{proposition} \label{prop:cosphUnique}
    The generalized Verma module $M_{-\lambda}$ is cospherical. Moreover, $\dim (M_{-\lambda}^{\mathsf{fin}*})^\kk = 1$.
\end{proposition}
\begin{proof}
    Note that for each constituent in (\ref{eqn:IDecomp}), we have 
    \[
    ((W_\mu\otimes W_\lambda^*)^*)^\kk \cong (W_\mu^* \otimes W_\lambda)^\kk \cong \Hom_\kk(W_\lambda, W_\mu).
    \]
    By Schur's Lemma, this space is one-dimensional if and only if $\mu = \lambda$. Thus we obtain a unique $\phi$ corresponding to $\mathrm{Id}_{W_\lambda} \in \Hom_\kk(W_\lambda, W_\lambda)$. We now define
    %For the degree 0 piece $W_\lambda\otimes W_\lambda^*$, we have $((W_\lambda\otimes W_\lambda^*)^*)^\kk \cong (W_\lambda^* \otimes W_\lambda)^\kk \cong \Hom_\kk(W_\lambda, W_\lambda)$ which is one dimensional by Schur's Lemma. So there exists a unique $\varphi \in ((W_\lambda\otimes W_\lambda^*)^*)^\kk$ corresponding to $\mathrm{Id}_{W_\lambda} \in \Hom_\kk(W_\lambda, W_\lambda)$. We now define
    \[
    \phi: M_{-\lambda} \xrightarrow[]{\mathrm{Proj}} W_\lambda \otimes W_\lambda^*\overset{\varphi}{\rightarrow}\C.
    \]
    Here the projection map $\mathrm{Proj}$ is defined with respect to (\ref{eqn:kFinDual}) which is a $\kk$-module decomposition. All maps involved are $\kk$-maps and $\phi \in M_{-\lambda}^{\mathsf{fin}*}$ is thus spherical and unique up to constant.
\end{proof}

By Definition~\ref{defn:ShimuraOp}, the Shimura operator $D_\mu \in \Uk^\kk$ can be defined explicitly as
\begin{equation}\label{eqn:Dmu}
    D_\mu := \sum_\ell \xi^{-\mu}_\ell \eta^\mu_\ell  
\end{equation}
where $\{\eta^\mu_\ell\}$ is a basis for $W_\mu \subseteq \Sk^{|\mu|}(\pplus)$ and its dual basis $\{\xi^{-\mu}_\ell\}$ for $W_\mu^*\subseteq \Sk^{|\mu|}(\pminus)$. 
%We would like to show the following vanishing action by $D_\mu$.

\begin{proposition} \label{prop:Dmu.by0}
    For all $\lambda \neq \mu$, $|\lambda| \leq |\mu|$, $D_\mu.\phi = 0$ on $M_{-\lambda}$. 
\end{proposition}

\begin{proof}
    Since $D_\mu\in \Uk^\kk$, this action is a scalar multiplication by Lemma~\ref{lem:scalar}, and thus it suffices to prove that $D_\mu.\phi$ vanishes on $W_\lambda \otimes W_\lambda^*$ (of degree 0) for all $\lambda \neq \mu$ with $l := |\lambda| \leq |\mu|$, that is,
    \begin{equation*}
        \langle W_\lambda \otimes W_\lambda^*, D_\mu.\phi \rangle = \sum \langle W_\lambda \otimes W_\lambda^*, \xi_\ell^{-\mu}\eta_\ell^\mu.\phi \rangle = \{0\}.
    \end{equation*}

%If $|\lambda| < |\mu|$, we may simply utilize the grading on $M_{-\lambda}$.
%Note $\xi_\ell^{-\mu} \in W_\mu^* \subseteq \Sk^{|\mu|}(\pminus)$ has degree $-|\mu|$. For each term appearing in the sum above, when we move the action of $\xi^{-\mu}_\ell$ contragrediently to the left side of the pairing, it lowers the degree 0 component $W_\lambda \otimes W_\lambda^*$ by $|\mu|$. Since the lowest degree of $M_{-\lambda}$ is $-|\lambda| > -|\mu|$, we get for every term.
If $|\mu| >l $, we use the following degree argument. Each $\xi_\ell^{-\mu}$, of degree $-|\mu|$, acts contragrediently on $W_\lambda \otimes W_\lambda^*$, of degree 0. By (\ref{eqn:degAct}) $\xi_\ell^{-\mu}W_\lambda \otimes W_\lambda^*$ is in $M_{-\lambda}^{-|\mu|} = \{0\}$. 
%But $-|\mu| < -l$ which is the lowest degree of $M_{-\lambda}$. 
Thus all pairings in the summation are 0.

%If $|\lambda| = |\mu|$, then $\xi_\ell^{-\mu}$ lowers $W_\lambda\otimes W_\lambda^*$ down to $M_{-\lambda}^{-|\mu|} = \C\otimes W_\lambda^*$. When we move $\eta_\ell^\mu$ to the left side, it raises $\C\otimes W_\lambda^*$ up into $W_\mu\otimes W_\lambda^*$. Although the contragredient actions produce signs, both $\phi\ev$ and $\phi\od$ vanish on $W_\mu\otimes W_\lambda^*$ whenever $\lambda \neq \mu$. Thus the pairing gives $0$. 
If however $|\mu| = l$, then the actions of $\xi_\ell^{-\mu}$ followed by $\eta_\ell^\mu$ move $W_\lambda\otimes W_\lambda^*$ up into $W_\mu\otimes W_\lambda^*$. But $\phi$ vanishes on $W_\mu\otimes W_\lambda^*$ whenever $\lambda \neq \mu$ by definition of $\phi$. Thus the pairing again gives $0$. 
\end{proof}

\subsection{The vanishing properties}\label{subsec:coupdegrâce}
%In this subsection, we prove the main result Theorem~\ref{main:RESULT} with an additional theorem on how $\Uk^\kk$ acts on the spherical functional depending on the highest weight of an irreducible module (Theorem~\ref{thm:scalar}).
%In what follows, we begin with an additional theorem on how $\Uk^\kk$ acts on the spherical functional depending on the highest weight of an irreducible module (Theorem~\ref{thm:scalar}), and then prove the main result Theorem~\ref{main:RESULT}.
Assuming Conjecture~\ref{conj:sph}, together with Proposition~\ref{prop:Sph=coSph}, we see that there exists a unique up to constant spherical functional $\kappa$ on $V_\lambda$. The next result relates $\kappa$ and $\phi$ on $M_{-\lambda}$ introduced above.

\begin{proposition} \label{prop:kappaphi}
    The spherical functional $\phi \in M_{-\lambda}^{\mathsf{fin}*}$ descends to $\kappa$ on $V_\lambda$, up to a constant.
\end{proposition}
\begin{proof}
    Let $U \subseteq M_{-\lambda}$ be the maximal submodule. Denote the projection $M_{-\lambda} \rightarrow M_{-\lambda}/U \cong V_\lambda$ by $\mathsf{proj}$. Define $\overline{\kappa} := \kappa \circ \mathsf{proj}$ on $M_{-\lambda}$. Since $\mathsf{proj}$ is a $\kk$-map,  $\overline{\kappa}$ is also spherical. 
    This $\overline{\kappa}$, by default, is in the full dual. 
    %But the codimension of $U$ is finite by Theorem~\ref{thm:CWWghtsFD}, implying that $\overline{\kappa}$ can only be non-zero on finitely many components in the decomposition (\ref{eqn:IDecomp}). 
    We prove that $\overline{\kappa}\in M_{-\lambda}^{\mathsf{fin}*}$. 
    
    Recall $ M_{-\lambda}^{-l+k}=\Sk^k(\pplus)\otimes W_\lambda^*.$ We first show that for $k > |\lambda| = l$, $E_k := \kk \Uk(\kk)M_{-\lambda}^{-l+k}$ is all of $ M_{-\lambda}^{-l+k}$. For the sake of contradiction, we pick a non-zero $T \in (M_{-\lambda}^{-l+k})^*$ that vanishes on $E_k$. Since $T(E_k) = 0$, we have $T(X.v) = 0$ for all $X \in \kk$ and $v\in M_{-\lambda}^{-l+k}$. So $T$ is cospherical. Thus $T \in \Hom_\kk (M_{-\lambda}^{-l+k}, \C) = \Hom_\kk (\Sk^k(\pplus)\otimes W_\lambda^*, \C) \cong \Hom_\kk(\Sk^k(\pplus), W_\lambda)$. But this space is zero whenever $k > |\lambda|$ by the Cheng--Wang decomposition (Proposition~\ref{prop:goodCW}). Contradiction. 
    
    Now as $E_k = M_{-\lambda}^{-l+k}$ for $k > l$, for any $v\in M_{-\lambda}^{-l+k}$, there exists $X \in \kk\Uk(\kk)$ and $v'\in M_{-\lambda}^{-l+k}$ such that $v = X.v'$. Hence $\overline{\kappa}(v) = \overline{\kappa}(X.v') = 0$, which implies that $\overline{\kappa}$ vanishes on all $M_{-\lambda}^{-l+k}$ for large $k$ ($k > l$). 
    %The module $M_{-\lambda}$ is a $\bk_\natural^-$-highest weight module. Notice that the Cartan subalgebra $\tk$ of $\bk_\natural^-$ is in $\kk$ and the decomposition (\ref{eqn:IDecomp}) is a $\kk$-decomposition. Thus, any highest weight vector in $M_{-\lambda}$ should sit entirely in a summand therein. The maximal submodule $M$ is generated by such a highest weight vector. Moreover, $M$ is of finite codimension by Theorem~\ref{thm:CWWghtsFD}. This means $M$ can only miss finitely many summands in $M_{-\lambda}$.
    Thus $\overline{\kappa}\in M_{-\lambda}^{\mathsf{fin}*}$. 
    Consequently, $\overline{\kappa}$ is a constant multiple of $\phi$ by Proposition~\ref{prop:cosphUnique}. Therefore, up to a constant, $\phi$ vanishes on $U$ and descends to $\kappa$ on $V_\lambda$.
\end{proof}

We now introduce a result that computes the eigenvalue of the action of $\Uk^\kk$ on modules. 
Assuming Conjecture~\ref{conj:sph}, then $V_\lambda$ is also cospherical by Proposition~\ref{prop:Sph=coSph}.

\begin{theorem}\label{thm:scalar}
If $D \in \Uk^\kk$, then $D$ acts on $\kappa \in (V_\lambda^*)^\kk$ by the scalar $\hcHomoGK(D)(2{\lambda^\natural}+\rho)$.
\end{theorem}
\begin{proof}
For an arbitrary $D\in \mathfrak{U^\kk}$, we write $D = uK+N^-u'+p \in \mathfrak{U}^\kk$ where $p= \pi(D)$ and $K\in \kk, N^-\in \mathfrak{n}^-, u, u'\in \mathfrak{U}$. Let us consider the pairing $\langle v, D.\kappa \rangle$, where $v$ is now of $\bk^-$-highest weight $-2\lambda_\hk^\natural$ (\ref{eqn:CWweightsOnA}) by Proposition~\ref{prop:even}.

On one hand, by Lemma~\ref{lem:scalar}, this pairing equals $C\langle v, \kappa \rangle$ for some constant $C$ depending on $D$.
On the other hand, we have
\begin{align*}
    \langle v, D.\kappa \rangle &= \langle v, uK.\kappa \rangle +\langle v,  N^- u'.\kappa \rangle + \langle v, p. \kappa \rangle && [D = uK+N^-u'+p ] \\
    &= \langle v, 0 \rangle +\langle 0, u'\kappa \rangle  + \langle v, p. \kappa \rangle && [\text{Contragredient action}]\\
    &=0+0+\langle v, p.\kappa \rangle .
\end{align*}
%\fixIt{Here the contragredient action is given by $\left \langle X.v, f \right \rangle = -(-1)^{|X||f|} \left \langle v, X.f \right \rangle$ for $v\in V$, $f\in V^*$ and $X\in \gk$}.

Let us write $p = \sum \prod a_i^{n_i}$ as a polynomial in $\mathfrak{S(a)} \cong \mathfrak{P}(\ak^*)$ (c.f. (\ref{eqn:identification})). Then $\langle v, p.\kappa \rangle = \langle \sum \prod (-1)^{n_i}a_i^{n_i}.v, \kappa \rangle$ as all $a_i$'s are even and commute. This gives
\begin{align*}
    \langle v, D.\kappa \rangle    &= \sum \prod \langle (-1)^{n_i}a_i^{n_i}.v, \kappa \rangle \\
    &= \sum\prod (-1)^{n_i} (a_i(-2\lambda_\hk^\natural|_\ak))^{n_i}\langle v, \kappa \rangle \\
    &= \sum \prod (-1)^{n_i}(-1)^{n_i}a_i(2{\lambda^\natural})^{n_i}\langle v, \kappa \rangle\\
    &= p(2{\lambda^\natural})\langle v, \kappa \rangle.
\end{align*}
Since $\langle v, \kappa \rangle \neq 0$ by Proposition~\ref{prop:non0pair}, the calculation yields $C = p(2{\lambda^\natural}) = \pi(D)(2{\lambda^\natural})$. Finally, (\ref{eqn:rhoshift}) says $C = \hcHomoGK(D) (2{\lambda^\natural}+\rho)$.
\end{proof}

\begin{theorem} \label{thm:vanCon}
For all $\lambda \neq \mu$,  $|\lambda|\leq |\mu|$, $\hcHomoGK(D_\mu)(2{\lambda^\natural}+\rho) = 0$, where $2{\lambda^\natural} = 2\lambda_\hk^\natural|_\ak$.
\end{theorem}

\begin{proof}
    By Proposition~\ref{prop:Dmu.by0}, $D_\mu$ acts on $\phi \in M_{-\lambda}^{\mathsf{fin}*}$ by 0. By Proposition~\ref{prop:kappaphi}, $\phi$ descends to (up to a constant) $\kappa$ on $V_\lambda$, and $\langle v, D_\mu.\phi \rangle = \langle v+M, D_\mu.\kappa \rangle = 0$ on $V_\lambda$.
    Therefore, $D_\mu$ acts on $\kappa$ by 0. By Theorem~\ref{thm:scalar}, we obtain the desired result.
\end{proof}
%{\color{red} Theorem~\ref{thm:vanCon} implies Conjecture \ref{conj:0} together with Theorem~\ref{thm:BCPoly}}

%\subsection{Proof of Theorem~\ref{main:RESULT}} \label{subsec:1.2proof}

Recall the Shimura operator $D_\mu$ (\ref{eqn:Dmu}) associated with $\mu$, can be written as $\sum \xi^{-\mu}_\ell \eta^\mu_\ell$.
In particular, we see that $\deg D_\mu \leq 2|\mu|$. Define a scalar $c_\mu$ so that
\[
c_\mu I_\mu(2{\mu^\natural}+\rho) = \hcHomoGK(D_\mu)(2{\mu^\natural}+\rho).
\]

\begin{proof}[Proof of Theorem~\ref{main:RESULT}]
In Proposition~\ref{prop:imGammaSupSym}, we see that the $\hcHomoGK$ image of $D_\mu$ does land in $\RingEv(\ak^*)$. By definition of $\hcHomoGK$, we have $\deg\hcHomoGK(D_\mu) \leq 2|\mu|$. Theorem~\ref{thm:vanCon} and the uniqueness of $I_\mu$ imply that $\hcHomoGK(D_\mu)$ is proportional to $I_\mu$.

We show that $\hcHomoGK(D_\mu)$ is non-zero. It suffices to prove that $D_\mu \notin \mathfrak{U}\kk$, which contains $\ker \hcHomoGK$ by Theorem~\ref{thm:prelHC}. 
We notice that $D_\mu$, by definition, is an element in $\mathfrak{S}(\pk^-)\mathfrak{S}(\pk^+)$. Since $\gk = \pminus \oplus \pplus \oplus \kk$, by the Poincaré--Birkhoff--Witt theorem, we may write
\[
\mathfrak{U} = \mathfrak{S}(\pk^-)\mathfrak{S}(\pk^+)\oplus \mathfrak{U}\kk 
\]
and since $D_\mu$ is clearly non-zero, $\hcHomoGK(D_\mu)$ is also non-zero. 
Finally, by definition of $c_\mu$, we see that $\hcHomoGK(D_\mu) = c_\mu I_\mu$ where $c_\mu\neq 0$.
\end{proof}

\begin{remark}\label{rmk:whyInsufficient}
We explain why the condition for sphericity given in \cite{a2015spherical} (Proposition~\ref{prop:AlldridgePart}) is insufficient. 
Consider the case $p=q=1$, and let $\mu$ be $(2)\in \mathscr{H}(1, 1)$. Thus, $|\lambda| \leq |\mu|$ means $\lambda = (1^n)$ for $n = 0, 1, 2$.
By (\ref{eqn:AiRho}), $\rho = (-1, 1)$, and $2{\lambda^\natural}+\rho = (1, 2n-1)$.
Then $I_\mu(1, 2n-1) = 0$ for $n = 0, 1, 2$.
By Proposition~\ref{prop:AlldridgePart}, the only partition in the form of $(1^n)$ guaranteed to give a spherical $V_\lambda$ is $\lambda = (1)$.
By the proof of Theorem~\ref{thm:vanCon} without assuming Theorem~\ref{main:11Sph} (Conjecture~\ref{conj:sph}), we see $\hcHomoGK(D_\mu)(1, 1)=0$.
But any degree 4 even supersymmetric polynomial in two variables is proportional to $f(x, y) = (x^2-y^2)(x^2+ay^2+b)$ and $f(1, 1)=0$ automatically. Thus $\hcHomoGK(D_\mu)$ cannot be pinned down by one single partition $\lambda = (1)$.
In fact, we have $\hcHomoGK(D_\mu)$ proportional to $(x^2-y^2)(x^2-1)$ which indeed vanishes at $(1, 2n-1)$ for all $n\geq 0$ (see the \textit{extra} vanishing properties in Appendix~\ref{app:GRS}).
\end{remark}

\section*{Acknowledgment}
The author would like to thank his advisor Siddhartha Sahi for his munificent help in proposing the project and providing insights; and to thank Shifra Reif for enlightening conversations. In the earlier stage of the project, the author also received great help from Hadi Salmasian. The author would also like to thank the anonymous referees who munificently helped to improve the presentation of this work.
\appendix

%\section{Realizations of Symmetric Superpairs} \label{app:supPairs}

%\subsection{Definitions and Notations} \label{app:basicDef}

%\subsection{Embeddings of the Superpairs}\label{app:emb}

%\subsection{The KKT Correspondence} \label{app:KKT}

\section{Deformed Root System and Interpolation Polynomials} \label{app:GRS}

%\subsection{Admissible Deformed Root Systems}
We introduce the deformed root systems studied by Sergeev and Veselov in \cites{SV2004DQCP, SV09}. These root systems are based on the generalized root systems introduced by Serganova \cite{S96GRS}.

Let $V = \mathbb{C}^{m+n}$ be a vector space with basis $\left \{ \epsilon_1 , \dots, \epsilon_{m+n} \right \}$. Set $I\ev :=\left \{ 1, \dots, m \right \}$, $I\od := \left \{ m+1, \dots, m+n \right \}$, and $I := I\ev \cup I\od$. Define $\delta_i = \epsilon_{m+i}$ for $1\leq i\leq n$. A bilinear form on $V$ is given by
\[
B(u, v) := \sum_{i\in I\ev} u_iv_i - \sum_{j\in I\od} u_jv_j
\]
where $x_i$ denote the $i$-th coordinate of $x= \sum_{i\in I} x_i\epsilon_i \in V$.
We say a vector $v$ is isotropic (respectively anisotropic) if $B(v, v) = 0$ (respectively $\neq 0$). 
Following \cite{SV2004DQCP}, we set
\begin{enumerate}
    \item $C(m, n)$: Let $C(I) :=\left \{ \pm 2\epsilon_i, \pm \epsilon_i \pm\epsilon_j : i, j\in I, i\neq j \right \}$. We define $\ani{\Sigma} :=  C(I\ev)\cup C(I\od)$ and $\iso{\Sigma} = \left \{\pm \epsilon_i\pm \epsilon_j :  i\in I\ev, j\in I\od \right \}$ for anisotropic and isotropic roots respectively. Then $C(m, n)$ is defined as $\ani{\Sigma} \cup \iso{\Sigma}$. %Obviously, $\Sigma\ev = C(m) \sqcup C(n)$ where $C(m)$ means the type $C$ root system of the usual Lie algebras.
    
    \item $BC(m, n)$: Let $BC(I)= \left \{\pm \epsilon_i, \pm 2\epsilon_i, \pm \epsilon_i \pm\epsilon_j : i, j\in I, i\neq j \right \}$. We define $\ani{\Sigma} :=  BC(I\ev)\cup BC(I\od)$ and $\Sigma\od = \left \{\pm \epsilon_i\pm \epsilon_j :  i\in I\ev, j\in I\od \right \}$ for anisotropic and isotropic roots respectively. Then $BC(m, n)$ is defined as $\ani{\Sigma} \cup \iso{\Sigma}$. 
    %Similarly, $\Sigma\ev = BC(m) \sqcup BC(n)$ where $BC(m)$ means the type $BC$ root system of the usual symmetric spaces.
\end{enumerate}
The associated Weyl group is of Type \textit{BC}, isomorphic to $(\mathscr{S}_m \ltimes (\mathbb{Z}/2\mathbb{Z})^m)\times (\mathscr{S}_n \ltimes (\mathbb{Z}/2\mathbb{Z})^n)$.
%When $n = 0$, we get the usual Type $C$ or Type $BC$ systems. 
In \cite{SV2004DQCP}, an admissible deformation of these generalized root systems is introduced. This appears in the study of the Calogero--Moser--Sutherland problem in which the root system stays the same while the bilinear form and the multiplicities are ``tweaked". The deformed bilinear form is
\[
B_\mathsf{k}(u, v) := \sum_{i\in I\ev} u_iv_i + \mathsf{k}\sum_{i\in I\od} u_iv_i.
\]
Five parameters regarding the multiplicities of the roots are presented as follows:
\begin{align*}
    &\mult\left(\pm\epsilon_i\pm \epsilon_j\right) = \mathsf{k}, \mult\left(\pm\epsilon_i\right) = \mathsf{p}, \mult\left(\pm 2\epsilon_i\right) = \mathsf{q}, \quad i, j\in I\ev;
\\
&\mult\left(\pm\delta_i\pm \delta_j\right) = \mathsf{k}^{-1}, \mult\left(\pm\delta_i\right) = \mathsf{r}, \mult\left(\pm 2\delta_i\right) = \mathsf{s} , \quad 1\leq i, j \leq n.
\end{align*}
They satisfy the following two relations:
\[
\mathsf{p = kr}, \mathsf{2q+1 = k(2s+1)}.
\]
The multiplicities of isotropic roots ($\pm \epu_i\pm \dau_j$) are set to be 1 (c.f. the parameter $q\od$ in Section~\ref{supPairs}), and the form and all multiplicities stay $W_0$ invariant. In particular, the restricted root system in Section~\ref{supPairs} is indeed a special case.

%\subsection{Interpolation Polynomials} \label{app:interP}
We now explain how we obtain our version of the Type $BC$ interpolation polynomials from the ones that appear in Section 6, \cite{SV09}. 

Let $\left \{ \epsilon_i , \delta_j \right \}$ be the standard basis for $\mathbb{C}^{p+q}$, $z_i$ and $w_j$ be the coordinates of $\epsilon_i$ and $\delta_j$ for $i = 1, \dots, p$, $j = 1,  \dots, q$.
Let $\mathsf{k}$ and $\mathsf{h}$ be two parameters. Following \cite{SV09}, we assume $\mathsf{k} \notin \mathbb{Q}_{>0}$ (called \textit{generic}). 
Define $\varrho = \sum\varrho_i\bos \epsilon_i + \sum \varrho_j\fer \delta_j$ where
\[
\varrho_i\bos := -(\mathsf{h}+\mathsf{k}i), \; \varrho_j\fer :=  -\mathsf{k}^{-1}\left(\mathsf{h}+\frac{1}{2}\mathsf{k}-\frac{1}{2}+j+\mathsf{k}p\right).
\]
\begin{Definition}[c.f. Definition~\ref{defn:noshiftBCSym}]
\label{defn:LambdaRing}
Let $P_{p, q} := \mathbb{C}[z_1, \dots, z_p, w_1, \dots, w_q]$ be the polynomial ring on $\mathbb{C}^{p+q}$. We define $\Lambda^\varrho$ to be the subalgebra of polynomials $f \in P_{p, q}$ which
\begin{enumerate}
    \item are symmetric separately in variables $(z_i-\varrho_i\bos)$ and $(w_j-\varrho_j\fer)$, and invariant under their sign changes;
    \item satisfy the condition $f(X-\epsilon_i+\delta_j) = f(X)$ on the hyperplane $z_i+\mathsf{k}(i-1-p) = \mathsf{k}w_j+j-1$.
\end{enumerate}

\end{Definition}

We equip $\mathbb{C}^{p+q}$ with an inner product defined by 
\[
(\epsilon_i, \epsilon_j) = \delta_{i, j}, \; (\delta_i, \delta_j) = \mathsf{k}\delta_{i, j}, \; (\epsilon_i, \delta_j) = 0
\]
where $\delta_{i, j}$ is the Kronecker delta. Then Condition (2) can be rephrased as:
\begin{enumerate}
    \item[(2')] $f(X+\alpha)=f(X)$ when $X$ is in the hyperplane defined by the equation $$(X-\varrho, \alpha) + \frac{1}{2}(\alpha, \alpha)=0$$ for all $\alpha = \epsilon_i-\delta_j$ for $i=1, \dots, p$, $j = 1, \dots, q$.
\end{enumerate}

The following table compares the notation used in \cite{SV09} and this paper. Since we will alter the constructions and proofs in \cite{SV09}, \textit{in what follows, we will use the notation given on the left.
}

\noindent\begin{tabu}{|X[c, m]|X[c, m]|} 
\hline
\cite{SV09} & Z.\\
\hline
$m, n$  & $p, q$   \\
\hline
$k, p, q, r, s, h$  & $\mathsf{k, p, q, r, s, h}$ \\
\hline
$m(-), \epu, \dau, R^+, \rho$ & $\mult(-), \epr, \dar, \Sigma^+,  \varrho$ \\
\hline
$\Lambda^{(k, h)}_{m, n}, w, z$ & $\Lambda^\varrho, x, y$ \\
\hline
\end{tabu}

The origin of the discrepancy lies in a different choice of positivity of the root system. Formula (71) in \cite{SV09} reads ``$R^+_{iso} = \left\{\delta_p\pm \epsilon_i \right\}$". That is, they require $\dar_j \pm \epr_i$ in this paper to be positive. 
The Weyl vector in \cite{SV09} is $\rho:= \frac{1}{2}\sum_{\alpha\in R^+} \mult(\alpha)\alpha$.

In the proof of \cite{SV09}*{Proposition~6.3}, they used the following restriction map $\mathrm{res}_{m, n}$ from $\Lambda^{(k, h)}$, the ring of symmetric functions of Type $BC$, to $P_{m, n} = \mathbb{C}[w_1, \dots, z_n]$.
For any $\lambda \in \mathscr{H}(m, n)$, let $\nu$ be the transpose of the first $n$ columns and $\mu$ the remaining part. 
%Notice that this is different from constructing $\lambda^\natural$ where we take the first $m$ rows and then take the remaining part. 
Let $\chi: \mathscr{H}(m, n) \rightarrow \mathbb{C}^{m+n}$ be the map which sends $\lambda$ to $(\mu_1, \dots, \mu_m, \nu'_1,\dots ,\nu'_n)$. The image of $\chi$ is Zariski dense and any symmetric function in $\Lambda^{(k, h)}$ can be restricted to a function on $\mathscr{H}(m, n)$. 
Denote this restriction map as $\mathrm{res}_{m, n}:\Lambda^{(k, h)} \rightarrow P_{m, n}$. More precisely, given $f \in \Lambda^{(k, h)}$, $\mathrm{res}_{m, n}(f)$ is the unique polynomial such that $f(\lambda) = \mathrm{res}_{m, n}(f)(\chi(\lambda))$ for any $\lambda \in \mathscr{H}(m, n)$.
Then \cite{SV09}*{Theorem~6.1} was shown using the even Bernoulli polynomials. Recall that if $k \notin \mathbb{Q}_+$, it is called generic. 

\begin{theorem}[\cite{SV09}*{Theorem~6.1}]\label{thm:res}
If $k$ is generic, then $\Ima \mathrm{res}_{m, n} = \Lambda_{m, n}^{(k, h)}$.
\end{theorem}

Proposition~6.3 in \cite{SV09} then follows from Theorem~\ref{thm:res} by applying $\mathrm{res}_{m, n}$ to the Okounkov interpolation polynomials in $\Lambda^{(k, h)}$. That is, 
\[
\textup{Sergeev--Veselov Interpolation Polynomial} = \mathrm{res}_{m, n}(\textup{Okounkov Polynomial})
\]
%Both are denoted using $I_\lambda$. 

The readers may recall that $\lambda^\natural$ is used in our paper consistently rather than $\chi(\lambda)$. This subtlety is merely a combinatorial consequence of different choices of positivity. We present how to adapt Theorem~\ref{thm:res} with our choice of positivity:

\begin{enumerate}
    \item In \cite{SV09}*{(71)}, change $R^+_{iso} = \left\{\delta_p\pm \epsilon_i \right\}$ to $R^+_{iso} = \left\{\epsilon_i\pm \delta_p \right\}$, so $\rho$ \cite{SV09}*{(72)} becomes
    \[
    -\sum_{i=1}^m (h+ki)\epsilon_i -k^{-1}\sum_{j=1}^n\left(h+\frac{1}{2}k-\frac{1}{2}+j+km\right)\delta_j.
    \]
    
    \item Consequently, the definition of the ring $\Lambda_{m, n}^{(k, h)}$ is now the subalgebra consisting of polynomials symmetric separately in
    \[
    (w_i+h+ki)^2 \text{ and }\left(z_j+hk^{-1}+\frac{1}{2}-\frac{1}{2}k^{-1}+jk^{-1}+m\right)^2
    \]
    and satisfying the same supersymmetry condition, changing $\alpha$ to $\epsilon_i-\delta_p$.
    
    \item For $\lambda$, bisect it by taking the first $m$ rows and the rest $n$ columns, so that we get $\lambda^\natural =: \chi(\lambda)$. Equivalently, $w_i = \lambda_i$ and $z_j = \langle \lambda_j'-m \rangle$. Again, this new version of $\chi$ has a Zariski dense image and $\mathrm{res}_{m, n}$ is well-defined. Here a partition $\lambda$ is regarded as a collection of boxes with coordinates $(i, j)$ for $1\leq i \leq \ell(\lambda)$ and $1\leq j \leq \lambda_i$.
    \item The statement of \cite{SV09}*{Theorem~6.1} remains verbatim, while in the proof, the polynomial $f_l^{m, n}$ is now changed to
    \begin{align*}
        f_l^{m, n}(w, z) = &\sum_{i=1}^m\left(B_{2l}\left(w_i+h+ki+\frac{1}{2} \right)-B_{2l}\left(h+ki+\frac{1}{2}\right) \right)\\
        &+k^{2i-1}\sum_{j=1}^n\left(B_{2l}\left( z_j+hk^{-1}-\frac{1}{2}k^{-1}+jk^{-1}+1+m\right) \right. \\
        & \left. -B_{2l}\left(hk^{-1}-\frac{1}{2}k^{-1}+jk^{-1}+1+m \right) \right).
    \end{align*}
    where $B_{2l}$ denotes the even Bernoulli polynomial (pp. 125-127 \cite{WW5ed})
    This version has value at $(w, z)=\lambda^\natural$ exactly equal to $f_l(\lambda)$. The rest of the proof is the same.
\end{enumerate}

Switching back to the notation in this paper, we see that $\mathrm{res}_{m, n}$ becomes $\mathrm{res}_{p, q}: \Lambda^{(\mathsf{k, h})}\rightarrow \Lambda^\varrho$.% in our context. %From this altered version of Theorem~\ref{thm:res}, we obtain:
By the above adaptation of Theorem~\ref{thm:res}, we obtain:

\begin{theorem}\label{thm:ExtraBCPoly}
For each $\mu \in \cH$, there exists a unique degree $2|\mu|$ polynomial $J_\mu \in \Lambda^\varrho$ such that
\begin{equation} \label{eqn:extraVan}
    J_\mu(\lambda^\natural; \mathsf{k, h}) = 0,  \quad \text{for all } \lambda \nsupseteq \mu
\end{equation}
and satisfies the normalization condition
\[
J_\mu(\mu^\natural; \mathsf{k, h}) = \prod_{(i, j)\in \mu}\left(\mu_i -j-\mathsf{k}(\mu'_j -i)+1 \right)\left(\mu_i +j+\mathsf{k}(\mu'_j +i)+2\mathsf{h}-1 \right).
\]
Moreover, they constitute a basis for $\Lambda^\varrho$.
\end{theorem}

It is not hard to see that for a degree $2|\mu|$ symmetric polynomial, the above vanishing properties (\ref{eqn:extraVan}), namely, all $\lambda$ that do not contain $\mu$, overdetermine $J_\mu$. Indeed, it is entirely possible to reduce this extra vanishing property.

Define $\Lambda^\varrho_d := \{f\in \RingEv: \deg f\leq 2d\}$, $\cH_d := \bigcup_{k\leq d} \cH^k$, and $2\cH_d^\natural := \left\{2{\lambda^\natural}: \lambda\in \cH_d\right \} \subseteq \C^{p+q}$.

\begin{proposition} \label{prop:resSurj}
    Every $f\in \Lambda^\varrho_d$ is determined by its values on $2\cH_d^\natural$.
\end{proposition}
\begin{proof}
    Let $\mathcal{V}_d$ be the space of functions on $2\cH_d^\natural$.
    Then $\dim \Lambda^\varrho_d = \dim \mathcal{V}_d= |\cH_d|$.
    In particular, $\mathcal{V}_d$ has a Kronecker-delta basis $\{\delta_{\lambda}: 
    \delta_\lambda(2{\lambda^\natural}) = 1, \delta_\lambda(2{\mu^\natural})=0, \lambda, \mu \in \cH_d \}$. 
    Next, the evaluation of $f\in \Lambda^\varrho_d$ on $\cH_d$ gives a restriction map
    \[
    \mathtt{res}: \Lambda^\varrho_d \rightarrow \mathcal{V}_d.
    \]
    To prove the statement, we show that $\mathtt{res}$ is an isomorphism. 

    Fix a total order $\succ$ on $\cH_d$ such that $\mu \succ \lambda$ implies $|\mu| \geq |\lambda|$. Let $R$ be the matrix  for $\mathtt{res}$ with respect to the bases $\{J_\mu\}$ for $\Lambda^\varrho_d$ and $\{\delta_\lambda\}$ for $\mathcal{V}_d$ arranged by $\succ$. 
    Since $J_\mu(2{\mu^\natural}) \neq 0$, and $J_\mu(2{\lambda^\natural}) = 0$ for any $\lambda$ such that $\mu \succ \lambda$, we see that $R$ is upper triangular with non-zero diagonal entries. 
    Therefore $R$ is invertible, proving the statement.
\end{proof}

As a direct consequence of Proposition~\ref{prop:resSurj}, we can reduce the extra vanishing properties (\ref{eqn:extraVan}) in Theorem~\ref{thm:ExtraBCPoly} to 
\begin{equation} \label{eqn:reducedVan}
    J_\mu(\lambda^\natural; \mathsf{k, h}) = 0,  \quad \text{for all } |\lambda| \leq |\mu|, \lambda \neq \mu.
\end{equation}

\begin{proof}[Proof of Theorem~\ref{thm:BCPolyRho}]
    We specify the parameters $\mathsf{k, h}$ for the ring $\Lambda^\varrho$ (Definition~\ref{defn:LambdaRing}) and $J_\mu$ in Theorem~\ref{thm:ExtraBCPoly} with the restricted root data of $\Sigma:= \Sigma(\gk, \ak)$ (see Section~\ref{supPairs}). 
    Then $\mathsf{k}=-1$. For $\mathsf{h} := -\mathsf{k}p-q-\frac{1}{2}\mathsf{p}-\mathsf{q}$ (following \cite{SV09}), we have $\mathsf{h} = p-q+\frac{1}{2}$.

    Also, $\varrho=-\frac{1}{2}\rho$ (see (\ref{eqn:AiRho})). Thus we consider $J_\mu$ in $\Lambda^{-\frac{1}{2}\rho}$ as in Theorem~\ref{thm:ExtraBCPoly} (with (\ref{eqn:reducedVan})). 
    Define the change of variables $\tau:\Lambda^{-\frac{1}{2}\rho} \rightarrow \mathfrak{P}(V)$ by $z_i \mapsto \frac{1}{2}(x_i-\rho_i\bos)$ and $w_j \mapsto \frac{1}{2}(y_j-\rho_j\fer)$. 
    Then $\tau$ preserves the ring structures and is bijective onto its image.
    Indeed, for $f\in \Lambda^{-\frac{1}{2}\rho}$, $\tau(f)$ is a polynomial symmetric in $\left(\frac{1}{2}(x_i-\rho_i\bos)+\frac{1}{2}\rho_i\bos\right)^2 = x_i^2$ and $\left(\frac{1}{2}(y_i-\rho_j\fer)+\frac{1}{2}\rho_j\fer\right)^2 = y_j^2$. 
    Also, $f \left(\mu\right) = f \left(\mu + \alpha\right)$ when $\left(\mu+\frac{1}{2}\rho, \alpha\right) = 0$ for $\alpha =\epsilon_i-\delta_j$. 
    Hence, $\tau(f)(X+e_i-d_j)=\tau(f)(X)$ when $x_i+y_j=0$, seen by substituting $X \mapsto 2\mu+\rho$. Therefore $\Ima(\tau) = \RingEv(V)$. 
    %The inverse $\tau^{-1}: \RingEv(\ak^*) \rightarrow \Lambda^{-\frac{1}{2}\rho}$ maps $x_i$ to $2z_i+\rho_i\bos$, and $y_j$ to $2w_j+\rho_j\fer$. 
    Theorem~\ref{thm:BCPolyRho} now follows by setting
    \begin{equation} 
        I_\mu (x_i, y_j) := \tau\left(J_\mu\left(z_i, w_j; -1, p-q-\frac{1}{2}\right)\right). \tag*{\qedhere}
    \end{equation}
\end{proof}

%\section{Some Details of the Restricted Root Systems}

\section{Section~\ref{sphRep} Computations} \label{app:11brackets}
%\subsection{Superbracket Tables}
In this appendix, we first give the table of superbrackets on $\gl(2|2)$ with the notation in Section~\ref{sphRep}, then show the detailed $\sltwo$-computations in the same section.

In the following table, the value of the superbracket of the $i$-th entry in the far left column and the $j$-th entry in the top row is given by the $(i, j)$-th entry.

\begin{center} \footnotesize
    \begin{tabular}{c|cccc cccc}
 & $\xi_{11}$  & $\xi_{12}$  & $\xi_{21}$  & $\xi_{22}$       & $X_1$      & $X_2$       & $Y_1$       & $Y_2$        \\ \hline
$\eta_{11}$    &$\dangle{1010}$&  $X_1$     & $Y_2$       & 0     & 0           & $\eta_{12}$ & $-\eta_{21}$ & 0               \\
$\eta_{12}$          & $X_2$        & 0           &$\dangle{1001}$& $X_1$   & 0           & 0           & $-\eta_{22}$  & $\eta_{11}$ \\
$\eta_{21}$          & $Y_1$        &$\dangle{0110}$& 0   & $Y_2$     & $-\eta_{11}$ & $\eta_{22}$& 0           & 0     \\
$\eta_{22}$          & 0            & $X_2$       & $Y_1$        & $\dangle{0101}$    & $-\eta_{12}$ & 0           & 0           & $\eta_{21}$  \\
$X_1$          & $-\xi_{12}$ & 0           & $-\xi_{22}$ & 0     & 0 & 0 & $\dangle{1 {(-1)} 0 0}$  & 0    \\
$X_2$          & 0           & 0           & $\xi_{11}$  & $\xi_{12}$  & 0 & 0 & 0 & $\dangle{001(-1)}$ \\
$Y_1$          & 0           & -$\xi_{11}$ & 0           & $-\xi_{21}$  & $\dangle{(-1)1 0 0}$ & 0  & 0 & 0\\
$Y_2$          & $\xi_{21}$  & $\xi_{22}$  & 0           & 0 & 0 & $\dangle{00(-1)1}$ & 0 & 0          
\end{tabular}
\end{center}

%\subsection{The Action on \texorpdfstring{$\Breve{W}$}{W}} 
Next, we present the $\sltwo$-computations. Let $\{x, h, y\}$ be an $\sltwo$-triple with the standard relations $[h, x]=2x, [h, y]=-2y, [x, y]=h$.
Let $L(m)$ be the $(m+1)$-dimensional irreducible $\sltwo$-module spanned by the weight vectors 
\begin{equation} \label{eqn:sl2Y}
\left \{ v_{-m},\dots , v_{-m+2i}, \dots, v_{m} \right \} \textup{ with } 
\begin{cases}
h.v_{-m+2i} = (-m+2i)v_{-m+2i} \\ 
x.v_{-m+2i} = v_{-m+2i+2} \\
y.v_{-m+2i} = i(m+1-i)v_{-m+2i-2}
\end{cases}.
\end{equation}
Since $\dim L(m) = m+1$, the scalar in the action of $y$ is $i(\dim L(m)-i)$.

Using the natural representation of $\gl(2)$, we write
\[
X = \begin{pmatrix}
0 & 1 \\
0 &  0\\
\end{pmatrix}, Y = \begin{pmatrix}
0 & 0 \\
1 & 0\\
\end{pmatrix}, H = \begin{pmatrix}
1 & 0 \\
0 & -1\\
\end{pmatrix}, Z = \begin{pmatrix}
1 & 0 \\
0 & 1\\
\end{pmatrix},
\]
which give a basis for $\gl(2)$ and $\{X, H, Y\}$ is an $\sltwo$-triple. Denote the standard characters on the diagonal Cartan algebra $\hk = \Spn\{H, Z\}$ as
\[
\epsilon^\pm: \begin{pmatrix}
x^+ & 0 \\
0 & x^-\\
\end{pmatrix}\mapsto x^\pm.
\]
%Consider $\Breve{W}(a, b, c, d)$. By definition, it is the irreducible $\gl(2)\oplus\gl(2)$-module $L(a, b)\otimes L(c, d)$ with the highest weight $a\epu^+ + b\epu^-+c\dau^++d\dau^-$. We extend the above $\sltwo$ convention to $\sltwo \oplus \sltwo$ and to $\gl(2)\oplus\gl(2)$ in the standard way, meaning $L(x, y)\cong L(x-y)$ as an $\sltwo$ module. Hence it has the highest (respectively lowest) weight $x-y$ (respectively $y-x$).
Let $a\geq b$ be two integers. Let $L(a, b)$ be the irreducible $\gl(2)$-module of highest weight $a\epsilon^++b\epsilon^-$ with respect to the standard Borel subalgebra $\Spn\{X, \hk\}$. 
A weight vector basis of $L(a, b)$ can be described by the set of symbols
\[
\{(b, a), (b+1, a-1), \dots, (a, b)\}
\]
such that $h.(m, n) = (m\epsilon^+ +n\epsilon^-)(h)(m, n)$ for $h\in \hk$ and $X.(m, n) = (m+1, n-1)$, whenever $(m+1, n-1)$ is a weight, and 0 otherwise. 
As an $\sltwo$-module, $L(a, b)$ is isomorphic to $L(a-b)$ by identifying $(b+i, a-i)$ with $v_{-(a-b)+2i}$ for $i = 0, 1, \dots, a-b$.

For $\gl(2)\oplus\gl(2)$, we use subscripts $i = 1, 2$ for $X, Y, H, Z$ in the $i$-th copy of $\gl(2)$. The standard characters on the second copy are denoted as $\delta^\pm$.
Consider the irreducible $\gl(2)\oplus\gl(2)$-module $L(a, b)\otimes L(c, d)$ with the highest weight $a\epu^+ + b\epu^-+c\dau^++d\dau^-$ with respect to the standard Borel subalgebra in both copies.
Denote it as $\Breve{W}(a, b| c, d)$.
Now each weight space in $\Breve{W}(a, b|c, d)$ is 1-dimensional, so there is no ambiguity by setting
\begin{enumerate}
    \item $(b, a| d, c)$ as the lowest weight vector $(b, a)\otimes(d, c)$ (identified as $v_{b-a}\otimes v_{d-c}$), 
    \item $X_1^kX_2^l.(b, a| d, c) = (b+k,a-k| d+l, c-l)$ (identified as $v_{b-a+2k}\otimes v_{d-c+2l}$).
\end{enumerate}
Then $Y_1, Y_2$ act accordingly. Moreover, $Y_iX_i$ acts as a constant determined by (\ref{eqn:sl2Y}) on any weight space.
The following is a detailed discussion that accounts for the proofs in Section~\ref{sphRep}.

%Its first component $L(a, -a+2)$ has the lowest weight $(-a+2, a)$, corresponding to $v_{-2a+2} \in L(2a-2)$. Note $-Y_1X_1.(1, 1, -1, -1) = -Y_1.(2, 0, -1, -1)$. The weight vector $(2, 0)$ corresponds to $v_{-2a+2+2a}$, and $y.v_{-2a+2+2a} = a(2a-1-a)v_0=(a^2-a)v_0$. So $-Y_1.(2, 0, -1, -1) = -(a^2-a)(1, 1, -1, -1)$.
    %Similarly $Y_2X_2.(1, 1, -1, -1) = Y_2.(1, 1, 0, -2) = (b+1)(2b+1-b-1)(1, 1, -1, -1)$. 
\begin{enumerate}
\item In the proof of Theorem~\ref{thm:b!=a-1}, the module is $\Breve{W}(a, -a+2 |b-1, -b-1) = L(a, -a+2)\otimes L(b-1, -b-1)$. We are interested in $Y_1X_1.(1, 1| -1, -1)$ and $Y_2X_2.(1, 1| -1, -1)$:
\begin{enumerate}
    \item In the first component $L(a, -a+2) \cong L(2a-2)$, $Y_1X_1.(1, 1| -1, -1)$ corresponds to 
    \[
    v_0\xmapsto{x}v_2 = v_{-2a+2+2a}\xmapsto{y} a(2a-2+1-a)v_0 = (a^2-a)v_0.
    \]
    \item In the second component $L(b-1, -b-1) \cong L(2b)$, $Y_2X_2.(1, 1| -1, -1)$ corresponds to 
    \[
    v_0\xmapsto{x}v_2 = v_{-2b+2(b+1)}\xmapsto{y} (b+1)(2b+1-b-1)v_0 = (b^2+b)v_0.
    \]
\end{enumerate}

Together, $(Y_2X_2-Y_1X_1)(1, 1| -1, -1) = (b^2+b-a^2+a) (1, 1| -1, -1)= (b+a)(b-a+1)(1, 1| -1, -1)$. Here, $(1, 1| -1, -1)$ represents $v$ in the proof.

\item In the proof of Theorem~\ref{thm:b=a-1}, the module is $\Breve{W}(a, -a+1| a-1, -a)$. We are interested in $Y_1X_1.(1, 0| -1, 0)$ and $Y_2X_2.(1, 0|-1, 0)$:
\begin{enumerate}
\item In the first component $L(a, -a+1) \cong L(2a-1)$, $Y_1X_1.(1, 0, -1, 0)$ corresponds to 
\[
v_1\xmapsto{x}v_3 = v_{-2a+1+2(a+1)}\xmapsto{y} (a+1)(2a-1+1-a-1)v_1 = (a^2-1)v_1.
\]
\item In the second component $L(a-1, -a) \cong L(2a-1)$, $Y_2X_2.(1, 0| -1, 0)$ corresponds to 
\[
v_{-1}\xmapsto{x}v_1 = v_{-2a+1+2a}\xmapsto{y} a(2a-1+1-a)v_{-1} = a^2v_{-1}.
\]
\end{enumerate}

Together, $(\dangle{0110}+ Y_2X_2- Y_1X_1).(1, 0| -1, 0) = 0$.
Here, $(1, 0| -1, 0)$ represents $v$ in the proof.
    
\end{enumerate}

\section{The Weyl Groupoid Invariance}\label{sec:HCiso}
%In this section, we define the Weyl groupoid $\mathfrak{W}$ associated with a root system. Also, its action is introduced here which was first considered by Sergeev and Vesolov in \cite{SV2011Grob}.

In this section, we provide a direct proof of the independence of $h_0$ in the description of $J_\alpha$ (\ref{eqn:Jalpha}), and a further digression on Weyl groupoid invariance studied by Sergeev and Veselov \cite{SV09}, and Serganova \cite{VSKacMoody}. We show that Proposition~\ref{prop:imGammaSupSym} is equivalent to a Weyl groupoid action formulation.

\begin{lemma} \label{lem:JAlphaIndep}
The subspace $J_\alpha$ is independent of choice of $h_0$.
\end{lemma}
\begin{proof}
Let $\{A_i\}$ be a basis for $\ak_\alpha^\perp$. %Let $\tau: \ak \rightarrow \mathbb{C}$ be the linear function $A\mapsto b(h_0, A)$.
Suppose $g_0 \neq h_0$ so that $b\left(g_0, g_0\right) = 0, \alpha(g_0) = 1$. Denote $\Spn_\mathbb{C} \{ g_0, t_\alpha \}$ as $U$. Let $\{B_i\}$ be a basis for $U^\perp$.
Let $J'_\alpha = \mathbb{C}\left [g_0^k t_\alpha^l  \right ] \Sk\left ( U^{ \perp} \right )$ as in (\ref{eqn:Jalpha}).
Given any monomial $g_0^k t_\alpha^l \prod B_i^{m_i} \in J'_\alpha$ with $l\geq \min\left \{ k, 1 \right \}$, we show that it is contained in $J_\alpha$. 

Suppose $g_0= ch_0+c_t t_\alpha + \sum c_j A_j$ for some $c, c_t$ and $c_j$. Note $\alpha(A_i) = b(t_\alpha, A_i) = 0$, and $\alpha (t_\alpha) = (\alpha, \alpha) = 0$. Applying $\alpha$ on $g_0$, we get $c = 1$, and
\begin{equation} \label{eqn:h0'}
    g_0= h_0 + c_t t_\alpha + \sum c_j A_j.
\end{equation}
Similarly, we write $B_i= C_ih_0 + D_i t_\alpha + \sum C_{ij} A_j$ for some $C_i, D_i$ and $C_{ij}$. Then $\alpha(B_i)=0$. Applying $\alpha$ on $B_i$, we get $C_i = 0$:
\begin{equation}\label{eqn:ai'}
    B_i= D_it_\alpha + \sum C_{ij} A_j.
\end{equation}
We expand $g_0^k t_\alpha^l \prod_{i = 1}^d B_i^{m_i}$ using (\ref{eqn:h0'}) and (\ref{eqn:ai'}) to see that the degree of $h_0$ is exactly $k$ while the degree of $t_\alpha$ is at least $l$.
Therefore, $J'_\alpha \subseteq J_\alpha$. The converse is similar.
\end{proof}

Following Sergeev and Vesolov \cite{SV2011Grob} (c.f. \cites{VSKacMoody, shifiWG}), we give the following definitions.
A category $\mathcal{C}$ is said to be \emph{small} if its collections of objects $\text{Obj}(\mathcal{C})$ and morphisms $\text{Hom}(\mathcal{C})$ are both sets. A \emph{groupoid} is a small category in which every morphism is invertible.
For two groupoids, the disjoint unions of their objects and morphisms constitute another groupoid.
Let $W$ be a group. An action of $W$ on a groupoid $\mathcal{G}$ is a group homomorphism from $W$ to $\mathrm{Aut}(\mathcal{G})$, the group of automorphisms (invertible endofunctors) of $\mathcal{G}$. 
Let $W$ be a group, $\mathcal{G}$ a groupoid and $W$ acts on $\mathcal{G}$.
The \emph{semidirect product} of $W$ and $\mathcal{G}$, denoted as $W \ltimes \mathcal{G}$, is given as follows:
\begin{enumerate}
        \item $\text{Obj}(W\ltimes \mathcal{G}) := \text{Obj}(\mathcal{G})$;
        \item $\Hom_{W \ltimes \mathcal{G}}(\alpha, \beta) := \{(\sigma, f): f \in \Hom_\mathcal{G}(\sigma (\alpha), \beta)\}$;
        \item For $(\sigma, f)\in \Hom_{W \ltimes \mathcal{G}}(\alpha, \beta)$ and $(\tau, g)\in \Hom_{W \ltimes \mathcal{G}}(\beta, \gamma)$, we define
        $(\tau, g)\circ (\sigma, f) := (\tau\sigma, g\circ \tau (f))
        \in \Hom_{W \ltimes \mathcal{G}}(\alpha, \gamma)$.
\end{enumerate}
Note that the composition is realized in $\mathcal{G}$ as in the diagram
$\tau\sigma(\alpha)\xrightarrow[]{\tau(f)}  \tau \beta  \xrightarrow[]{g}  \gamma$.

Suppose $\Sigma \subseteq V$ is a generalized root system as in \cite{S96GRS}. 
Denote the set of isotropic roots as $\iso{\Sigma}$.
Let $W_0$ be the Weyl group which acts on $\iso{\Sigma}$ naturally. 
The \emph{isotropic roots groupoid} $\zeroS$ is a groupoid with objects $\mathrm{Obj}(\zeroS) = \iso{\Sigma}$, and with non-trivial morphisms $\Bar{\tau}_\alpha:\alpha \rightarrow -\alpha$. Thus 
\[
\Hom_\zeroS(\alpha, \beta) = \begin{cases}
\varnothing & \text{ if } \beta \neq \pm\alpha \\
\{\Bar{\tau}_\alpha\} & \text{ if } \beta = -\alpha \\
\{\mathrm{id}_\alpha \} & \text{ if } \beta = \alpha 
\end{cases}.
\]

%The Weyl group $W_0$ associated to $\Sigma$ also acts naturally on $\iso{\Sigma}$ by permuting $\alpha\bnf$ and changing their signs. 
To extend the Weyl group action to morphisms in $\zeroS$, we set $\sigma (\Bar{\tau}_\alpha) = \Bar{\tau}_{\sigma(\alpha)}$ and $\sigma (\mathrm{id}_\alpha) = \mathrm{id}_{\sigma(\alpha)}$ for $\sigma \in W_0$.
We view $W_0$ as a groupoid with one single object $\ast$ whose morphisms are elements in $W_0$.
Then the \emph{Weyl groupoid} $\mathfrak{W}$ is defined as the disjoint union
    \[
    \mathfrak{W} := W_0 \sqcup  W_0 \ltimes \zeroS.
    \]

Let $V$ be a vector space. The \textit{affine groupoid} $\mathcal{AF}(V)$ is defined as the category whose objects are $V$ and all its affine subspaces, and its morphisms are all the affine isomorphisms between the affine subspaces of $V$.

\begin{Definition}\label{def:fcninv}
Let $\mathcal{G}$ be a groupoid and $V$ a vector space. We say $\mathcal{G}$ \textit{acts on} $V$ if there is a functor $\mathtt{C}: \mathcal{G} \rightarrow \mathcal{AF}(V)$.
For $F \in \Pk(V)$, we say $F$ is \textit{$\mathcal{G}$-invariant} if for all $x \xrightarrow{f} y$ in $\mathcal{G}$, we have $F\vert_{\mathtt{C}(x)} = F\vert_{\mathtt{C}(y)}\circ \mathtt{C}(f)$.
%$\mathtt{C}(f)$ intertwines with $F\vert_{\mathtt{C}(x)}$ and $F\vert_{\mathtt{C}(y)}$, that is,
    % \[
    %     F\vert_{\mathtt{C}(x)} = F\vert_{\mathtt{C}(y)}\circ \mathtt{C}(f).
    % \]
If $\mathcal{G}$ has only one object $\ast$ and $\mathtt{C}(\ast) = V$, then this degenerates to the usual definition of group invariance. We denote the set of $\mathcal{G}$-invariant polynomials as $\Pk(V)^\mathcal{G}$.
\end{Definition}

We specialize $\Sigma = \Sigma(\gk, \ak)$ and $V = \ak^*$. Thus $W_0$ is of Type \textit{BC}.
In \cite{SV2011Grob}, an action of $\mathfrak{W}$ on $V$ is described. Recall $\Hyp_\alpha = \{\mu\in V: (\mu, \alpha)=0\}$ and set $\tau_\alpha: \Hyp_\alpha \rightarrow \Hyp_\alpha$ by $\mu\mapsto \mu+\alpha$.
The action is given by the functor $\mathtt{I}:\mathfrak{W}\rightarrow \mathcal{AF}(\ak^*)$ that maps $\ast\in \mathrm{Obj}(W_0)$ to $\ak^*$, $\alpha \in \mathrm{Obj}( W_0 \ltimes \zeroS) = \iso{\Sigma}$ to $\Hyp_\alpha$, and 
\begin{align*}
    \mathtt{I}(\sigma) &= (\sigma:\ak^*\rightarrow \ak^*),\quad \sigma\in \Hom_{W_0}(\ast, \ast) \\
    \mathtt{I}((\sigma, \Bar{\tau}_{\alpha})) &= (\tau_{\alpha} \circ \sigma: \Hyp_{\sigma^{-1}(\alpha)}\rightarrow \Hyp_{\alpha}),\quad (\sigma, \Bar{\tau}_{\alpha})\in \Hom_{W_0 \ltimes \zeroS}(\sigma^{-1}(\alpha), -\alpha) \\
    \mathtt{I}((\sigma, \mathrm{id}_{\alpha})) &= (\mathrm{Id}_{\alpha}\circ \sigma: \Hyp_{\sigma^{-1}(\alpha)}\rightarrow \Hyp_{\alpha}),\quad (\sigma, \mathrm{id}_{\alpha})\in \Hom_{W_0 \ltimes \zeroS}(\sigma^{-1}(\alpha), \alpha).
\end{align*}
%We see $\mathtt{I}$ is well-defined as $\Hyp_\alpha = \Hyp_{-\alpha}$, and $(\mu, \sigma^{-1}\alpha)=0$ implies $(\sigma\mu, \alpha)=0$, so $\sigma$ maps $\Hyp_{\sigma^{-1}\alpha}$ to $\Hyp_\alpha$ for any $\sigma\in W_0$ and $\alpha\in \iso{\Sigma}$.
%We note $((\sigma', \Bar{\tau}_{-\sigma' \alpha})\circ (\sigma, \Bar{\tau}_{\alpha})) = (\sigma'\sigma, \Bar{\tau}_{-\sigma'\alpha}\circ \sigma' \Bar{\tau}_{\alpha}) = (\sigma'\sigma, \mathrm{id}_{\sigma'\alpha})$ in $\Hom_{W_0\ltimes \zeroS}(\sigma^{-1}\alpha, \sigma'\alpha)$. To see $\mathtt{I}$ preserves composition, it is enough to check that $\tau_{-\sigma'\alpha}\circ\sigma'\circ \tau_\alpha \circ \sigma = \mathrm{Id}_{\sigma'\alpha}\circ \sigma'\sigma$ which is a direct computation. It is similar for those compositions involving $(\sigma, \mathrm{id}_\alpha)$.

%Finally, we can check that the even supersymmetry conditions satisfied by the symmetric polynomials defined in Definition~\ref{defn:noshiftBCSym} is precisely the Weyl groupoid invariance in the sense of Definition~\ref{def:fcninv} where the action is given by the above $\mathtt{I}$:

\begin{proposition} \label{prop:WGinv}
We have $\Ima \hcHomoGK = \mathfrak{P}(\ak^*)^{\mathfrak{W}}$.
\end{proposition}
\begin{proof}
By Proposition~\ref{prop:imGammaSupSym}, we have $\Ima \hcHomoGK = \RingEv(\ak^*)$. Also, $\mathfrak{P}(\ak^*)^{\mathfrak{W}} = \mathfrak{P}(\ak^*)^{W_0}\cap \mathfrak{P}(\ak^*)^{W_0\ltimes \zeroS}$.

Indeed, Condition (i) in Definition~\ref{defn:noshiftBCSym}, the usual even symmetry, is captured by the natural action of $W_0$ by permuting $x_i$ and $y_j$. Thus $F$ satisfying Condition (i) is equivalent to $F \in \mathfrak{P}(\ak^*)^{W_0}$.

Condition (ii') (specified from (ii) to our root system) says $F\vert_{\Hyp_\alpha}(X) = F\vert_{\Hyp_\alpha}(X+\alpha)$ for any $\alpha \in \iso{\Sigma}$. The $W_0\ltimes \zeroS$ invariance means $F\vert_{\mathtt{I}(\sigma^{-1}(\alpha))}(X) = F\vert_{\mathtt{I}(-\alpha)}\left(\mathtt{I}((\sigma, \Bar{\tau}_{\alpha}))X\right)$ for any $X\in \Hyp_{\sigma^{-1}(\alpha)}$. We observe that the right hand side is $F\vert_{\Hyp_\alpha}(\sigma X+\alpha)$. By specifying $\sigma = 1$, we recover Condition (ii'). If $F$ satisfies both (i) and (ii'), then one may take $X \in \Hyp_{\sigma^{-1}\alpha}$ so $F\vert_{\Hyp_{\sigma^{-1}(\alpha)}}(X) = F\vert_{\Hyp_{\sigma^{-1}(\alpha)}}(X+\sigma^{-1}(\alpha))$. But the right side is equal to $F\vert_{\Hyp_{(\alpha)}}(\sigma X+\alpha)$ by Condition (i). 

Thus, $F\in  \RingEv(\ak^*)$ is equivalent to $F\in \mathfrak{P}(\ak^*)^{\mathfrak{W}}$. 
\end{proof}

%Weyl groupoid has a unique distinguished object corresponding to the finite Weyl group sub-groupoid. In fact, $\mathcal{S}(X)$, $\mathcal{GL}(V)$ and $\mathcal{AF}(V)$ all have such an object (set $X$ and vector space $V$ respectively).

%To define an action of $\mathfrak{W}$ on a vector space $V$, one may either specify a functor from $\mathfrak{W}$ to $\mathcal{AF}(V)$, or, following Remark~\ref{rmk:faithful}, assign an affine subspace $\Hyp_\alpha$ to each $\alpha$, an affine isomorphism between $\Hyp_\alpha$ and $\Hyp_{-\alpha}$, then define the action of $W_0$ on them, i.e. a homomorphism $\phi: W_0 \rightarrow \text{Aut}(\mathcal{AF}(V))$. Of course, the distinguished object goes to the whole space.

    %This is indeed an action of $\mathfrak{W}$ by Remark \ref{rmk:faithful} because the assignments here are all injective.

 %<-------------------

% Your bilbigraphy           %<-------------------

\bibliography{ref}

\end{document}